\pdfoutput=1
 \documentclass[11pt, a4paper]{article}
 \usepackage{graphicx}
 \usepackage{amsthm,amsmath,amssymb}
 \usepackage{stackrel}
 \usepackage{ulem}

\numberwithin{equation}{section}
\newtheorem{thm}{Theorem}[section]
\newtheorem{Lem}{Lemma}[section]
\newtheorem{defn}{Definition}[section]
\newtheorem{rem}{Remark}[section]
\newtheorem{exam}{Example}[section]

\newtheorem{Corollary}{Corollary}[section]

\begin{document}
\begin{center}
	\LARGE \textbf{Grundy Numbers of Impartial Chocolate Bar Games}
\end{center}

\begin{center}
	\large Ryohei Miyadera \footnote{Kwansei Gakuin, runners@kwansei.ac.jp}, 
	Shunsuke Nakamura \footnote{Osaka University, nakashun1994@gmail.com}
	Yushi Nakaya \footnote{Kwansei Gakuin, math271k@gmail.com}, 
\end{center}

\begin{abstract}
Chocolate bar games are variants of the CHOMP game in which the goal is to leave your opponent with the single bitter part of the chocolate. The original chocolate bar game \cite{robin} consists of a rectangular bar of chocolate with one bitter corner. Since the horizontal and vertical grooves are independent, an $m \times n$ rectangular chocolate bar is equivalent to the game of NIM with a heap of $m-1$ stones and a heap of $n-1$ stones. Since the Grundy number of the game of NIM with a heap of $m-1$ stones and a heap of $n-1$ stones is $(m-1) \oplus (n-1)$, the Grundy number of this $m \times n$ rectangular bar is $(m-1) \oplus (n-1)$.

In this paper, we investigate step chocolate bars whose widths are determined by a fixed function of the horizontal distance from the bitter square. 

When the width of chocolate bar is proportional to the distance from the bitter square and the constant of proportionality is even, the authors have already proved that the Grundy number of this chocolate bar is  $(m-1) \oplus (n-1)$, where $m$ is is the largest width of the chocolate and $n$  is the longest horizontal distance from the bitter part. This result was published in a mathematics journal (Integers, Volume 15, 2015).

On the other hand, if the constant of proportionality is odd, the Grundy number of this chocolate bar is not  $(m-1) \oplus (n-1)$.

Therefore, it is natural to look for a necessary and sufficient condition for chocolate bars to have the Grundy number that is equal to  $(m-1) \oplus (n-1)$, where $m$ is the largest width of the chocolate and $n$  is the longest horizontal distance from the bitter part.

In the first part of the present paper, the authors present this  necessary and sufficient condition.

Next, we modified the condition that the Grundy number that is equal to  $(m-1) \oplus (n-1)$, and we studied a  necessary and sufficient condition for chocolate bars to have Grundy number that is equal to $((m-1) \oplus (n-1+s))-s$, where $m$ is is the largest width of the chocolate and $n$  is the longest horizontal distance from the bitter part. We present this necessary and sufficient condition in the second part of this paper.
\end{abstract}

\section{Introduction}\label{firstsection}
The original chocolate bar game \cite{robin} had a rectangular bar of chocolate with one bitter corner. Each player in turn breaks the bar in a straight line along the grooves and eats the piece he breaks off.  The player who breaks the chocolate bar and eats to leave his opponent with the single bitter block (black block) is the winner. 

The $4 \times 3$ rectangular chocolate bar in Figure \ref{introchoco1} is equivalent to the game of NIM with a heap of 3 stones and a heap of 2 stones. Here, 
3 is the number of grooves above and 2 is the number of grooves to the right of the bitter square. 
Since the Grundy number of the game of NIM with a heap of $3$ stones and a heap of $2$ stones is $3 \oplus 2$, the Grundy number of the  chocolate bar in Figure \ref{introchoco1} is $3 \oplus 2$.

In this paper, we consider step  chocolate bars as in Figures \ref{introchoco2}, \ref{introchoco3}, \ref{introchoco4},\ref{jgamek1},\ref{jgamek3}, \ref{jgamek5} and etc.
In these cases, a vertical break can reduce the number of horizontal breaks. We can still think of the game as being played with heaps but now a move may change more than one heap.

\begin{rem}
If we are to play a non-trivial chocolate game, we have to make the \textit{disjunctive sum} of games by combining chocolates in Figures \ref{introchoco2}, \ref{introchoco3}, \ref{introchoco4},\ref{jgamek1},\ref{jgamek3},\ref{jgamek5} with another chocolate. For the definition of 
\textit{disjunctive sum}, see Definition \ref{sumofgames}.
For example, if we combine the chocolate in Figure \ref{choco2r} with the chocolate in Figure \ref{choco2l}, then we have the chocolate in Figure \ref{choco2rl} that is mathematically the same as the chocolate in  Figure \ref{samechoco}.
The chocolate game in Figure \ref{choco2r} itself is trivial, since the first player wins the game by cutting the brown parts instantly.

On the other hand, we can find a winning stratedy of the chocolate game in Figure \ref{samechoco} by studying the chocolate game  in Figure \ref{choco2r} and the chocolate game in Figure \ref{choco2l} separately.
\end{rem}

\begin{figure}[!htb]
	\begin{minipage}[!htb]{0.45\columnwidth}	
		\centering	
		\includegraphics[width=0.6\columnwidth,bb=0 0 187 93]{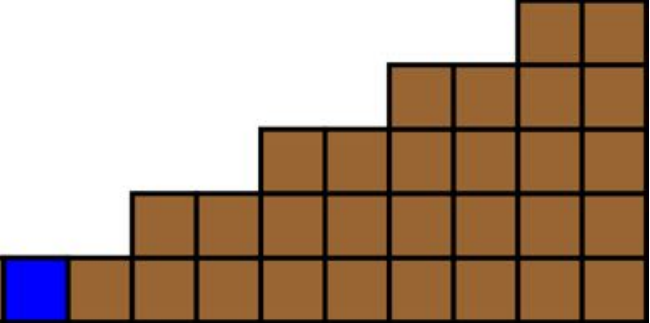}
		\caption{example $(1)$}
		\label{choco2r}
\end{minipage}
	\begin{minipage}[!htb]{0.45\columnwidth}
		\centering	
		\includegraphics[width=0.4\columnwidth,bb=0 0 76 51]{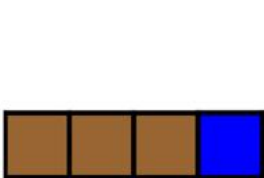}
		\caption{example $(2)$}
		\label{choco2l}
	\end{minipage}
\end{figure}

\begin{figure}[!htb]
	\begin{minipage}[!htb]{0.45\columnwidth}
		\centering	
		\includegraphics[width=0.8\columnwidth,bb=0 0 324 93]{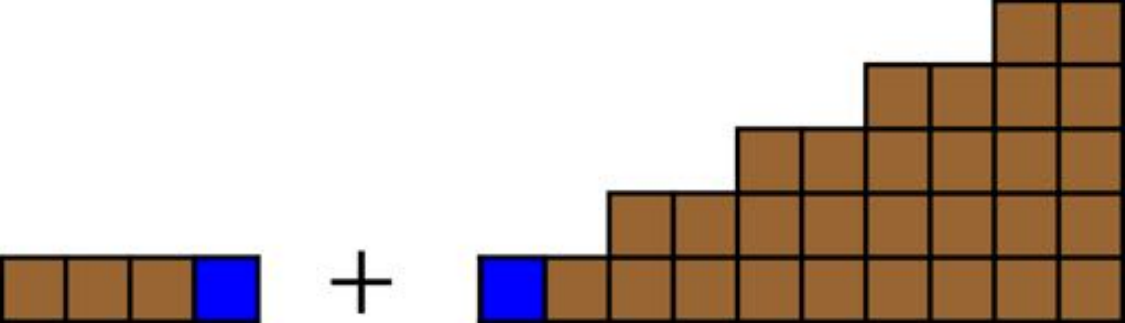}
		\caption{Figure \ref{choco2r} + Figure \ref{choco2l} (1)}
		\label{choco2rl}
	\end{minipage}
	\begin{minipage}[!htb]{0.45\columnwidth}	
		\centering	
		\includegraphics[width=0.8\columnwidth,bb=0 0 242 93]{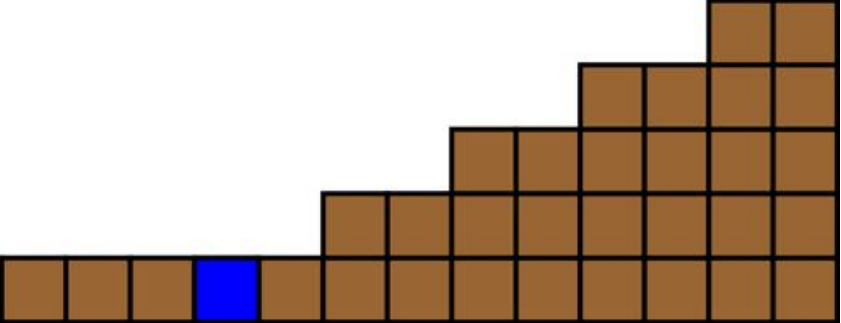}
		\caption{Figure \ref{choco2r} + Figure \ref{choco2l} (2)}
		\label{samechoco}
	\end{minipage}
\end{figure}

Next, we define coordinates \{y,z\} for each chocolate. Let $y=m-1$ and $z=n-1$,  where $m$ is is the largest width of the chocolate and $n$  is the longest horizontal distance from the bitter part. (We use $x$-coodinate later in this paper.)

The $y$-coordinate and the $z$-cooridinate of each step chocolate bar in Figures  \ref{introchoco2}, \ref{introchoco3}, \ref{introchoco4}, \ref{jgamek1},   
 \ref{jgamek3}, \ref{jgamek5} are $\{3,13\}$, $\{4,9\}$, $\{2,9\}$, $\{5,5\}$, $\{3,11\}$ and $\{2,14\}$. Here, the first number is the $y$-coordinate and the second number is the $z$-coordinate of each chocolate. For an example, the $y$-coordinate and the $z$-cooridinate of the chocolate bar in Figure \ref{introchoco2} are $3$ and $13$ respectively.

If we caluculate the Grundy numbers of chocolate bars in Figures  \ref{introchoco2}, \ref{introchoco3} and \ref{introchoco4}, we have $14$=$3 \oplus 13$, $13$=$4 \oplus 9$ and $11$=$2 \oplus 9$. (Here, we omit the calculation of Grundy numbers.)

If we caluculate the Grundy numbers of chocolate bars in Figures \ref{jgamek1}, \ref{jgamek3} and \ref{jgamek5}, then we have $8 \neq 0=5\oplus 5$, $13 \neq 8=3\oplus 11$ and $15 \neq 12=2\oplus 14$. (Here, we omit the calculation of Grundy numbers.)
Now we know that the Grundy number of some step chocolates are $y \oplus z$, where $y$ and $z$ are coordinates of the chocolate, but 
the Grundy number of some step chocolates are not $y \oplus z$.

When the width of chocolate bar is proportional to the distance from the bitter square and the constant of proportionality is even, the authors have already proved that the Grundy number of this chocolate bar is $y \oplus z$, where $y,z$ are the coordinates of the chocolate bar. This result was published in a mathematics journal~\cite{integer2015}. Chocolates in Figures  \ref{introchoco2}, \ref{introchoco3} and \ref{introchoco4} are examples of this type of chocolates.

On the other hand, if the constant of proportionality is odd, the Grundy number of this chocolate bar is not $y \oplus z$. Chocolates in Figures \ref{jgamek1}, \ref{jgamek3} and \ref{jgamek5} are examples of this type of chocolates.

Therefore it is natural to look for a necessary and sufficient condition for chocolate bars to have the Grundy number that is equal to  $y \oplus z$, where $y$ and $z$ are coordinates of the chocolate.

There are other types of chocolate bar  games, and one of the most well known is CHOMP. CHOMP is a game with a rectangular chocolate bar. The players take turns, and they choose one block and eat it together with those that are below it and to its right. The top left block is bitter and the players cannot eat this block. Although many people have studied this game, the winning strategy is yet to be discovered. For an overview of research of CHOMP, see  \cite{zeilberger}.

\begin{figure}[!htb]
	\begin{minipage}[!htb]{0.45\columnwidth}
		\centering	
		\includegraphics[width=0.2\columnwidth,bb=0 0 44 59]{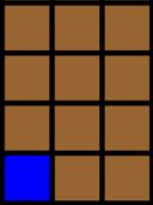}
		\caption{$\{3,2\}$ Grundy number is 1=$3 \oplus 2$.}
		\label{introchoco1}
	\end{minipage}
	\begin{minipage}[!htb]{0.45\columnwidth}	
		\centering	
		\includegraphics[width=0.8\columnwidth,bb=0 0 203 58]{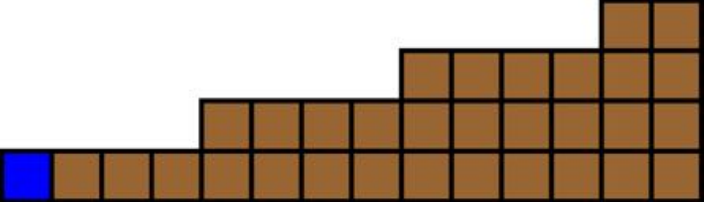}
		\caption{$\{3,13\}$ Grundy number is 14=$3 \oplus 13$.}
		\label{introchoco2}
	\end{minipage}
\end{figure}

\begin{figure}[!htb]
	\begin{minipage}[!htb]{0.45\columnwidth}
		\centering	
		\includegraphics[width=0.45\columnwidth,bb=0 0 261 133]{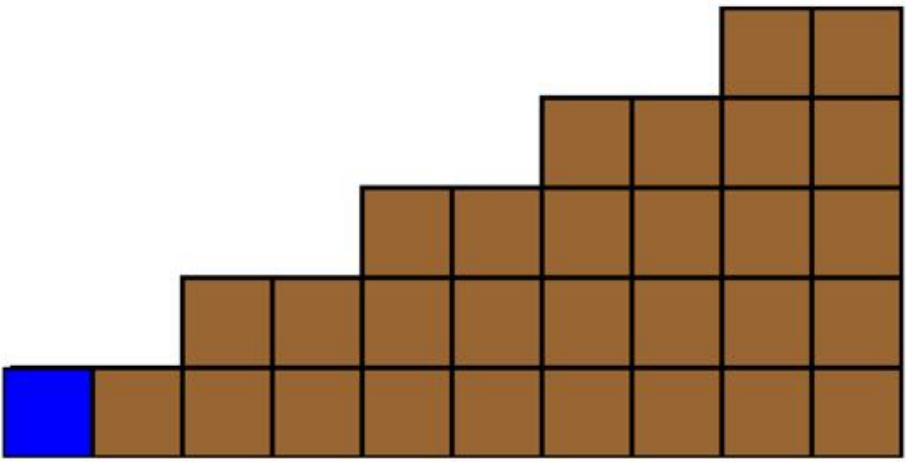}
		\caption{$\{4,9\}$ Grundy number is 13=$4 \oplus 9$.}
		\label{introchoco3}
	\end{minipage}
	\begin{minipage}[!htb]{0.45\columnwidth}	
		\centering	
		\includegraphics[width=0.45\columnwidth,bb=0 0 261 79]{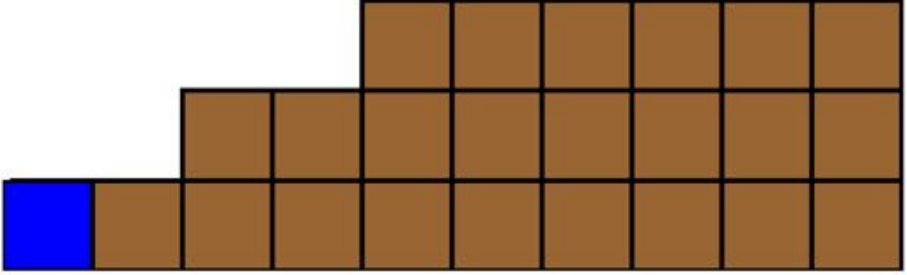}
		\caption{$\{2,9\}$  Grundy number is 11=$2 \oplus 9$.}
		\label{introchoco4}
	\end{minipage}
\end{figure}

\begin{figure}[!htb]
	\begin{minipage}[!htb]{0.32\columnwidth}		
		\centering	
		\includegraphics[width=0.45\columnwidth,bb=0 0 260 260]{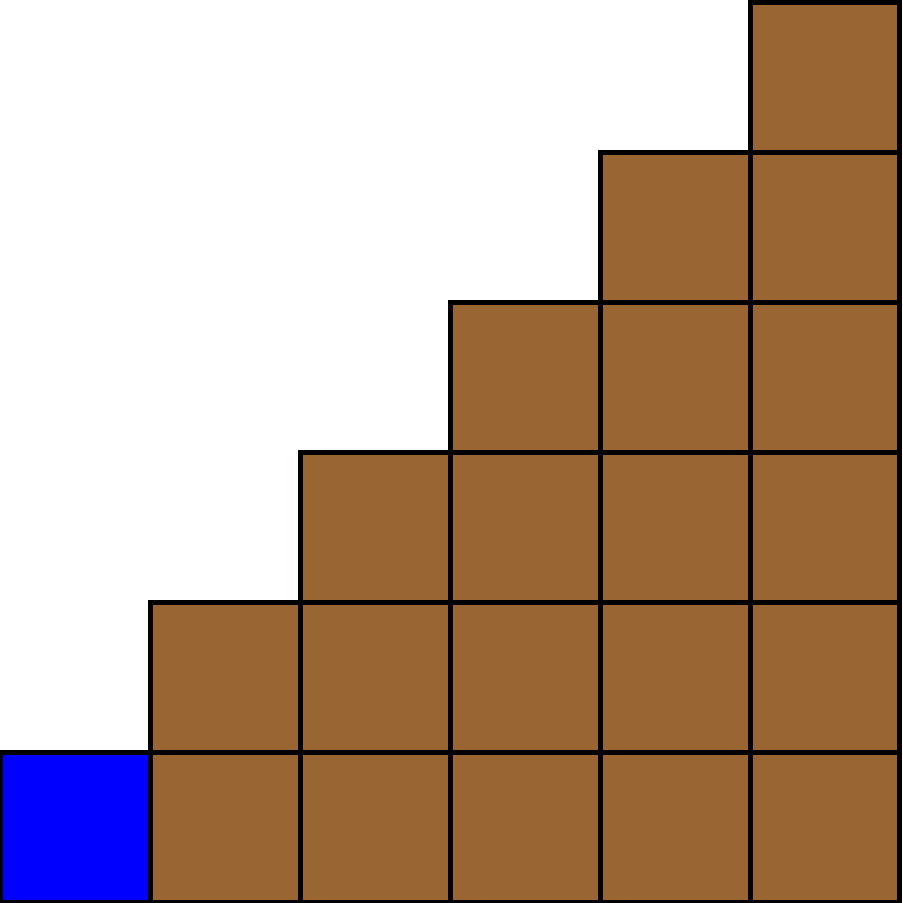}
		\caption{$\{5,5\}$ Grundy number is $8 \neq 0=5\oplus 5$.}
		\label{jgamek1}
	\end{minipage}
	\begin{minipage}[!htb]{0.32\columnwidth}
		\centering	
		\includegraphics[width=0.6\columnwidth,bb=0 0 260 87]{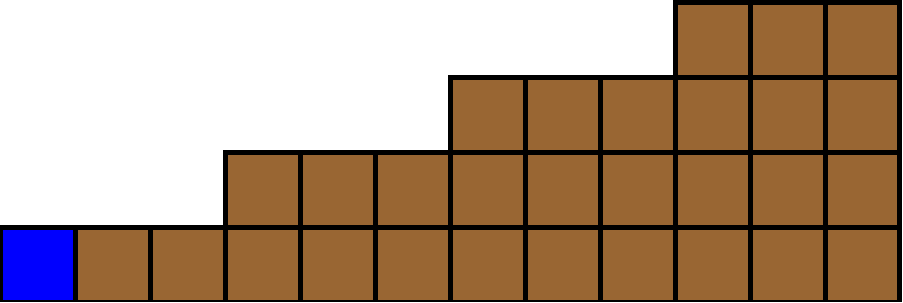}
		\caption{$\{3,11\}$ Grundy number is $13 \neq 8=3\oplus 11$.}
		\label{jgamek3}
	\end{minipage}
	\begin{minipage}[!htb]{0.32\columnwidth}
		\centering	
		\includegraphics[width=0.7\columnwidth,bb=0 0 260 52]{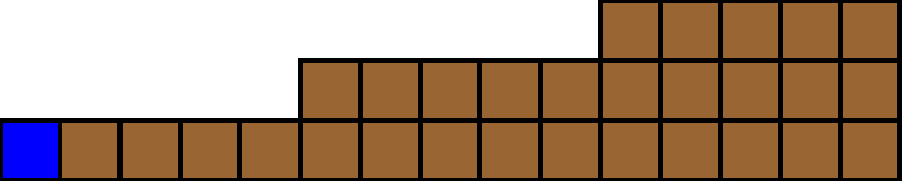}
		\caption{$\{2,14\}$ Grundy number is $15 \neq 12=2\oplus 14$.}
		\label{jgamek5}
	\end{minipage}
\end{figure}

Let $Z_{\geq0}$ be the set of non-negative integers. 

For completeness, we briefly review some necessary concepts in combinatorial game theory; see  $\cite{lesson}$ or $\cite{combysiegel}$ for more details. 

\begin{defn}\label{definitionfonimsum11}
	Let $x$, $y$ be non-negative integers, and write them in base 2, so 
	that $x = \sum_{i=0}^n x_i 2^i$ and $y = \sum_{i=0}^n y_i 2^i$ with $x_i,y_i \in \{0,1\}$.
  We  define the nim-sum $x \oplus y$ by
 \begin{equation}
x \oplus y = \sum\limits_{i = 0}^n {{w_i}} {2^i},
\end{equation}
where $w_{i}=x_{i}+y_{i} \ (mod\ 2)$.
 \end{defn}
 
Since chocolate bar games are impartial games without draws there will only be two outcome classes.
\begin{defn}\label{NPpositions}
$(a)$ $\mathcal{N}$-positions, from which the next player can force a win, as long as he plays correctly at every stage.\\
$(b)$ $\mathcal{P}$-positions, from which the previous player (the player who will play after the next player) can force a win, as long as he plays correctly at every stage.
\end{defn}

\begin{defn}\label{sumofgames}
The \textit{disjunctive sum} of two games, denoted $\mathbf{G}+\mathbf{H}$, is a super-game where a player may move either in $\mathbf{G}$ or in $\mathbf{H}$, but not both.
\end{defn}

\begin{defn}\label{defofmove}
For any position $\mathbf{p}$ of a game $\mathbf{G}$, there is a set of positions that can be reached by making precisely one move in $\mathbf{G}$, which we will denote by \textit{move}$(\mathbf{p})$. 
\end{defn}

\begin{rem}
As to the examples of \textit{move}, please see Example \ref{examofmove11}.
\end{rem}

\begin{defn}\label{defofmexgrundy}
$(i)$ The \textit{minimum excluded value} ($\textit{mex}$) of a set, $S$, of non-negative integers is the least non-negative integer which is not in S. \\
$(ii)$  Each position $\mathbf{p}$ of a impartial game has an associated Grundy number, and we denoted it by $G(\mathbf{p})$.\\
Grundy number is found recursively: 
$G(\mathbf{p}) = \textit{mex}\{G(\mathbf{h}): \mathbf{h} \in move(\mathbf{p})\}.$
\end{defn}

We present some lemmas for \textit{minimum excluded value} $mex$.
\begin{Lem}\label{defofmex1}
Let $x \in Z_{\geq 0}$ and $y_k \in Z_{\geq 0}$ for $ k=1,2,...,n$. Then  Condition $(i)$ is true if and only if Condition $(ii)$ is true.\\
$(i)$   $x = \textit{mex}(\{y_k,k=1,2,...,n\})$. \\
$(ii)$   $x \neq y_k$ for any $k$ and for any $u \in Z_{\geq 0}$ such that $u < x$ there exists $k$ such that
$u = y_k$.
\end{Lem}
\begin{proof}
This is direct from Definition \ref{defofmexgrundy} (the definition of $mex$).
\end{proof}

\begin{Lem}\label{defofmex2}
Let $S$ be a set and $\{1,2,3,...,m-1\} \subset S$. Then $\textit{mex}(S) \geq m$.
\end{Lem}
\begin{proof}
	This is direct from the definition of $mex$ in Definition \ref{defofmexgrundy}.
\end{proof}

The power of the Sprague-Grundy theory for impartial games is contained in the next result.

\begin{thm}\label{theoremofsumg}
Let $\mathbf{G}$ and $\mathbf{H}$ be impartial games, and let $G_{\mathbf{G}}$ and $G_{\mathbf{H}}$ be Grundy numbers of $\mathbf{G}$ and $\mathbf{H}$  respectively. Then we have the following:\\
$(i)$ For any position $\mathbf{g}$ of $\mathbf{G}$ we have
$G_{\mathbf{G}}(\mathbf{g})=0$ if and only if $\mathbf{g}$ is a $\mathcal{P}$-position.\\
$(ii)$ The Grundy number of a position $\{\mathbf{g},\mathbf{h}\}$ in the game $\mathbf{G}+\mathbf{H}$ is
$G_{\mathbf{G}}(\mathbf{g})\oplus G_{\mathbf{H}}(\mathbf{h})$.
\end{thm}
For a proof of this theorem, see $\cite{lesson}$.

\section{Grundy Numbers of chocolate bar}
In this paper we study Grundy numbers of chocolate bar.
For a general bar, the strategies seem complicated. We focus on bars that grow regularly in height. 
\begin{defn}\label{definitionoffunctionf}
Let $f$ be a function that satisfies the following two conditions:\\
$(i)$  $f(t)\in Z_{\geq0}$ for $t \in Z_{\geq0}$.\\
$(ii)$  $f$ is monotonically increasing,i.e.,
we have $f(u) \leq f(v)$ for $u,v \in Z_{\geq0}$ with $u \leq v$.
 \end{defn}

\begin{defn}\label{defofbarwithfunc}
Let $f$ be the function that satisfies the conditions in Definition \ref{definitionoffunctionf}.\\ For $y,z \in Z_{\geq0}$ 
the chocolate bar will consist of $z+1$ columns where the 0th column is the bitter square and the height of the $i$-th column is
$t(i) = \min (f(i),y) +1$ for i = 0,1,...,z. We will denote this by $CB(f,y,z)$.\\
Thus the height of the $i$-th column is determined by the value of $\min (f(i),y) +1$ that is determined by $f$, $i$ and $y$.
\end{defn}

\begin{exam}\label{exampofcbars}
Here are examples of chocolate bar games $CB(f,y,z)$. 

\begin{figure}[!htb]
	\begin{minipage}[!htb]{0.32\columnwidth}
		\centering
		\includegraphics[width=0.8\columnwidth,,bb=0 0 203 58]{referee1.pdf}
		\caption{$CB(f,3,13)$ $f(t)$  $= \lfloor \frac{t}{4}\rfloor$.}
		\label{referee1p}
	\end{minipage}
	\begin{minipage}[!htb]{0.32\columnwidth}
		\centering
		\includegraphics[width=0.8\columnwidth,bb=0 0 261 133]{k2right.pdf}
		\caption{$CB(f,4,9)$ $f(t)$  $= \lfloor \frac{t}{2}\rfloor$.}
		\label{referee2p}
	\end{minipage}
	\begin{minipage}[!htb]{0.32\columnwidth}
		\centering
		\includegraphics[width=0.8\columnwidth,bb=0 0 261 79]{k2rightb.pdf}
		\caption{$CB(f,2,9)$ $f(t)$  $= \lfloor \frac{t}{2}\rfloor$.}
		\label{referee3p}
	\end{minipage}
\end{figure}

\begin{figure}[!htb]
	\centering
	\includegraphics[width=0.6\columnwidth,bb= 0 0 1238 350]{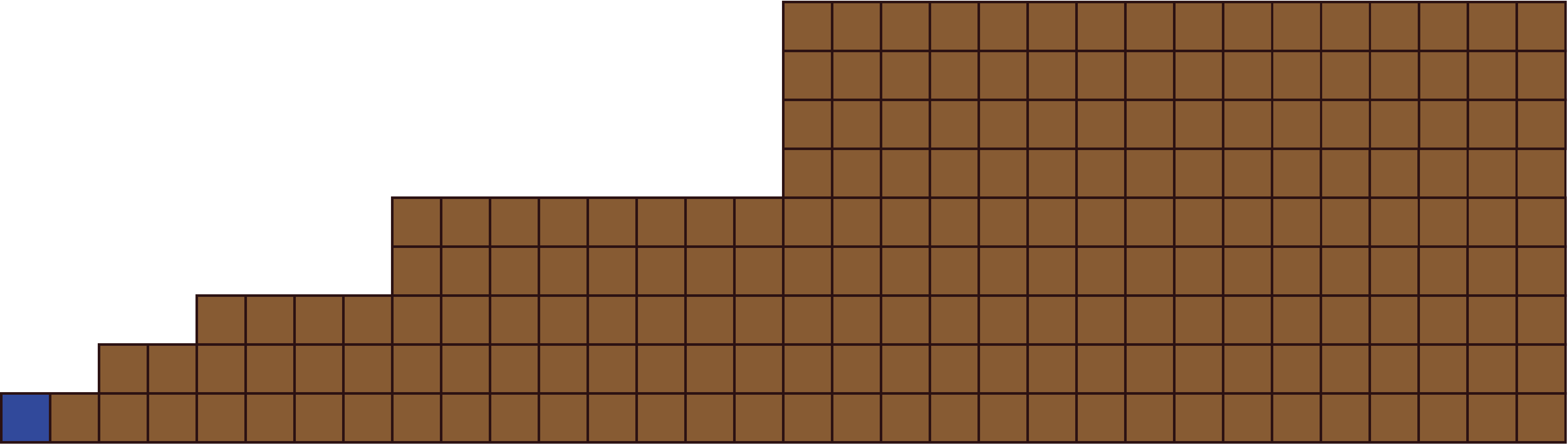}
	\label{yzk4h2graph}
	\caption{$CB(f,8,31)$ $f(0)=f(1)=0$ and $f(t)$  $=2^ {\lfloor log_2t \rfloor -1}$ for $t > 1$.}
\end{figure}

\end{exam}

For a fixed function $f$,  we denote the position of $CB(f,y,z)$ by coordinates $\{y,z\}$ without mentioning $f$.

\begin{exam}\label{examplecoordinates}
Here, we present four examples of coordinates of positions of chocolate bars  when $f(t)$  $= \lfloor \frac{t}{2}\rfloor$.

\begin{figure}[!htb]
	\begin{minipage}[!htb]{0.45\columnwidth}
		\centering
		\includegraphics[width=0.5\columnwidth,bb=0 0 196 98]{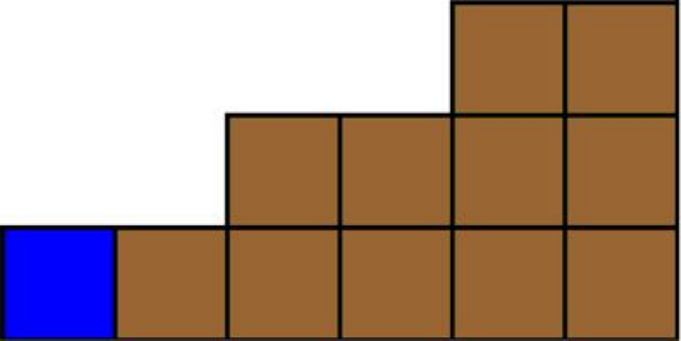}
		\caption{$CB(f,2,5)$ $\{2,5\}$}
		\label{2yzchoco} 
	\end{minipage}
	\begin{minipage}[!htb]{0.45\columnwidth}
		\centering
		\includegraphics[width=0.5\columnwidth,bb=0 0 174 87]{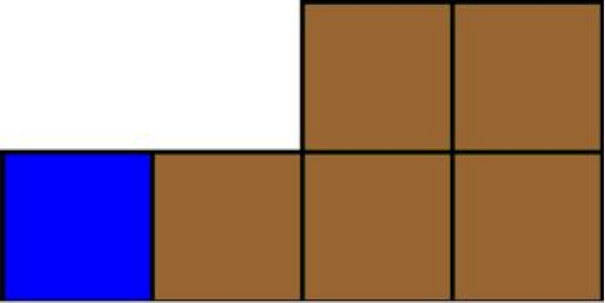}
		\caption{$CB(f,1,3)$ $\{1,3\}$}
		\label{minex3}
	\end{minipage}
\end{figure}

\begin{figure}[!htb]
	\begin{minipage}[!htb]{0.45\columnwidth}
		\centering
		\includegraphics[width=0.45\columnwidth,bb=0 0 196 33]{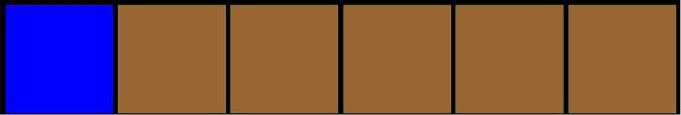}
		\caption{$CB(f,0,5)$ $\{0,5\}$}
		\label{minex1v2}
\end{minipage}
	\begin{minipage}[!htb]{0.45\columnwidth}
		\centering
		\includegraphics[width=0.45\columnwidth,bb=0 0 260 87]{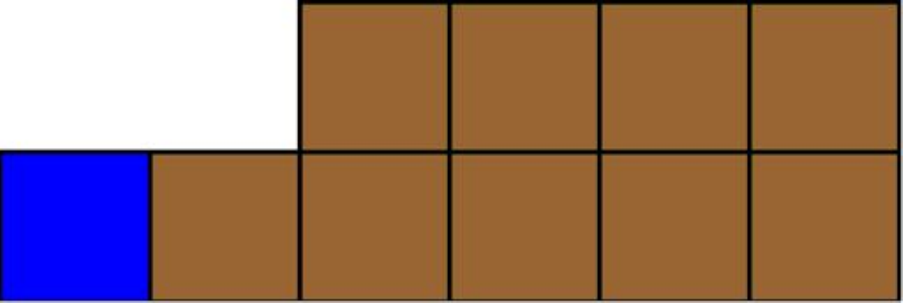}
		\caption{$CB(f,1,5)$ $\{1,5\}$}
		\label{minex1v2b}
	\end{minipage}
\end{figure}
\end{exam}

For a fixed function $f$, we define $move_f$ for each position $\{y,z\}$ of the chocolate bar $CB(f,y,z)$. This $move_f$ is a special case of $move$ defined in Definition \ref{defofmove}.
\begin{defn}\label{moveofchocoh}
	For $y,z \in Z_{\ge 0}$ we define \\
	$move_f(\{y,z\})=\{\{v,z \}:v<y \} \cup  \{ \{\min(y, f(w)),w \}:w<z \}$, where $v,w \in Z_{\ge 0}$.
\end{defn}

\begin{exam}\label{examofmove11}
Here, we explain about move when $f(t)$  $= \lfloor \frac{t}{2}\rfloor$.
 If we start with the position $\{y,z\}=\{2,5\}$ and reduce $z=5$ to $z=3$, then the y-coordinate (the first coordinate)  will be $\min(2, \lfloor 3/2 \rfloor )=\min(2,1)=1$.\\ Therefore we have $\{1,3\} \in move_f(\{ 2,5 \})$. It is easy to see that 
$\{1,5\}, \{0,5\} \in move_f(\{2,5\})$, $\{1,3\} \in move_f(\{1,5\})$ and $\{0,5\} \notin move_f(\{1,3\})$. 
\end{exam}

\section{A Chocolate Game Bar $CB(f,y,z)$ whose Grundy Number is $G(\{y,z\})=y \oplus z$}\label{thirdsection}
\subsection{A Sufficient Condition for a Chocolate Bar $CB(f,y,z)$ to have the Grundy Number $G(\{y,z\})=y \oplus z$}\label{subsectionforsufficondi}
In this subsection we study a sufficient condition for a chocolate bar $CB(f,y,z)$ to have a Grundy number $G(\{y,z\}) = y \oplus z$. 

In our proofs, it will be useful to have the disjunctive sum of a chocolate bar $CB(f,y,z)$ to the right of the bitter square and a single strip of chocolate bar  to the left, as in Figures  \ref{choco20151}, \ref{choco20152}, \ref{choco20153},  \ref{2yzchoco1}, \ref{demochocolate3} and \ref{demochocolate4}. We will denote such a position by $\{x,y,z\}$, where $x$ is the length of the single strip of chocolate bar and $y,z$ are coordinates of $CB(f,y,z)$. Figures \ref{2yzchoco1}, \ref{demochocolate3} and \ref{demochocolate4} give some examples of the coordinate system. 

For the disjunctive sum of the chocolate bar  game with $CB(f,y,z)$ to the right of the bitter square and a single strip of chocolate bar  to the left, we will show that the $\mathcal{P}$-positions are when $x\oplus y\oplus z=0$, so that the Grundy number of the chocolate bar $CB(f,y,z)$ is $x$ $ =y\oplus z$. 

\begin{exam}
	Examples of coordinates of chocolate bar games.

\begin{figure}[!htb]
	\begin{minipage}[!htb]{0.45\columnwidth}
		\centering
		\includegraphics[width=0.7\columnwidth,bb=0 0 260 122]{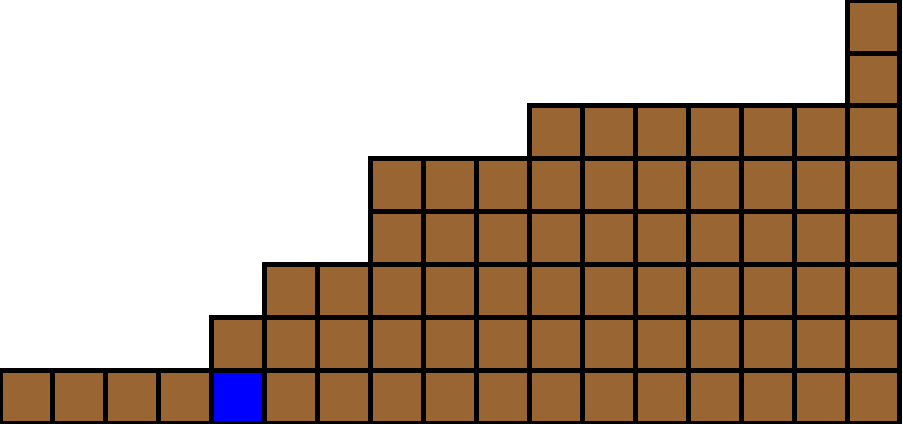}
		\caption{$\{4,7,12\}$}
		\label{choco20151}
	\end{minipage}
	\begin{minipage}[!htb]{0.45\columnwidth}
		\centering
		\includegraphics[width=0.7\columnwidth,bb=0 0 260 112]{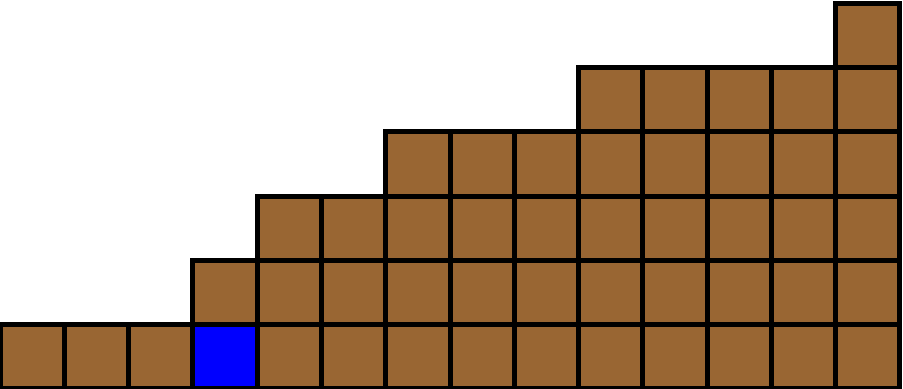}
		\caption{$\{3,5,10\}$}
		\label{choco20152}
	\end{minipage}
	\end{figure}
	
\begin{figure}[!htb]
	\begin{minipage}[!htb]{0.45\columnwidth}
		\centering
		\includegraphics[width=0.5\columnwidth,bb=0 0 260 186]{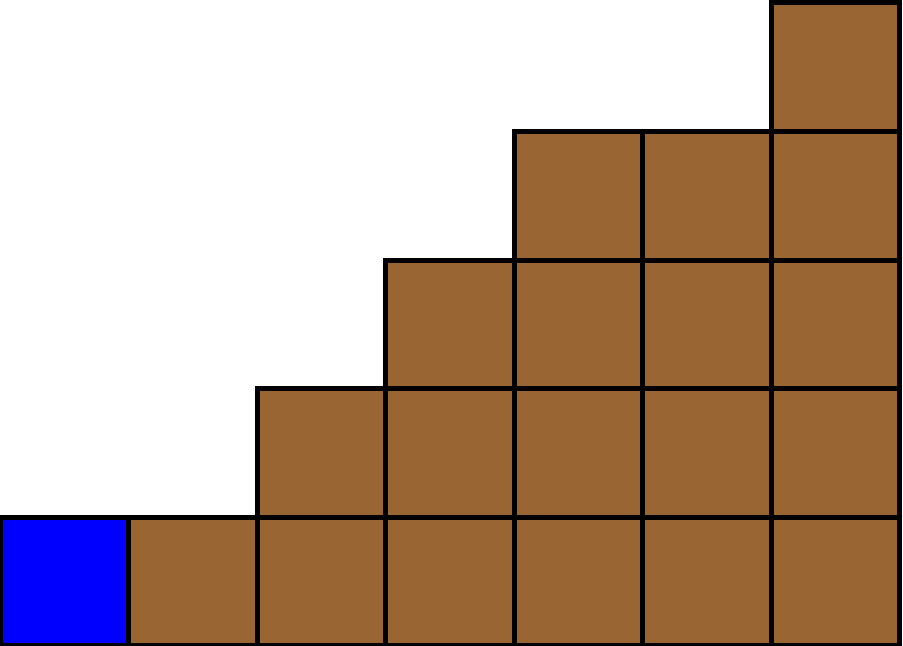}
		\caption{$\{0,4,6\}$}
		\label{choco20153}
	\end{minipage}
	\begin{minipage}[!htb]{0.45\columnwidth}
		\centering
		\includegraphics[width=0.7\columnwidth,bb=0 0 242 93]{demochocol2.pdf}
		\caption{$\{3,4,9\}$}
		\label{2yzchoco1}
	\end{minipage}	
\end{figure}
	
\begin{figure}[!htb]
	\begin{minipage}[!htb]{0.45\columnwidth}
		\centering
		\includegraphics[width=0.8\columnwidth,bb=0 0 260 58]{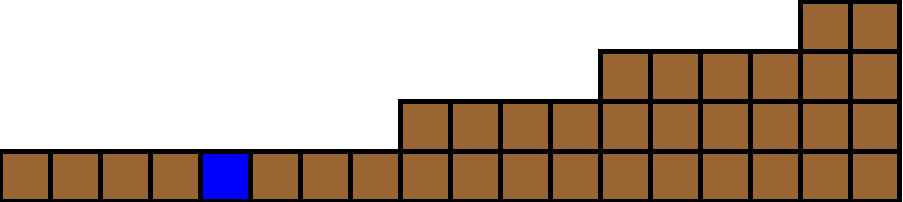}
		\caption{$\{4,3,13\}$}
		\label{demochocolate3}
	\end{minipage}
	\begin{minipage}[!htb]{0.45\columnwidth}
		\centering
		\includegraphics[width=0.8\columnwidth,bb=0 0 529 132]{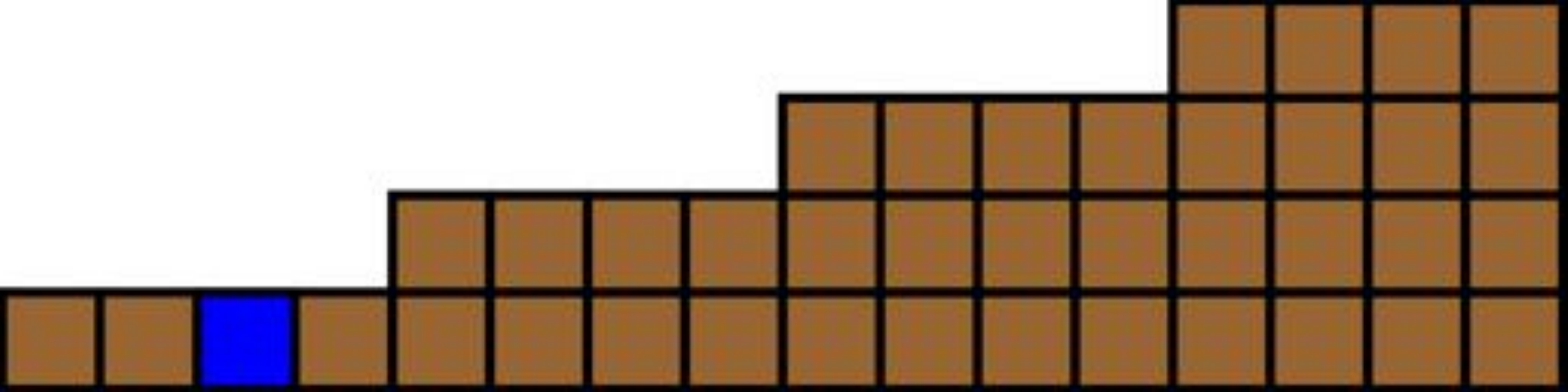}
		\caption{$\{2,3,13\}$}
		\label{demochocolate4}
	\end{minipage}
\end{figure}

\end{exam}

 $move_h (\{x, y, z\})$ is the set that contains all the positions that can be reached from the position $\{x, y, z\}$ in one step (directly).

\begin{defn}\label{moveforthreecordinates}
	For $x,y,z \in Z_{\ge 0}$, we define \\
	$move_h(\{x,y,z\})=\{\{u,y,z \}:u<x \} \cup \{\{x,v,z \}:v<y \} \cup \{ \{x,\min(y, h(w) ),w \}:w<z \}$, where $u,v,w \in Z_{\ge 0}$.
\end{defn}

The following condition $(a)$ in Definition \ref{adefoffunch} is a sufficient condition for a chocolate bar $CB(f,y,z)$ to have a Grundy number $G(\{y,z\}) = y \oplus z$. 

\begin{defn}\label{adefoffunch}
Let $h$ be a function of $Z_{\geq0}$ into $Z_{\geq0}$ that satisfies the conditions of Definition \ref{definitionoffunctionf} and the following condition $(a)$.\\
$(a)$ Suppose that 
\begin{equation}\label{aaconditionforh}
\lfloor \frac{z}{2^i}\rfloor = \lfloor \frac{z^{\prime}}{2^i}\rfloor
\end{equation}
for some $z,z^{\prime} \in Z_{\geq 0}$ and some natural number $i$.
Then we have 
\begin{equation}\label{aaconditionforh2}
\lfloor \frac{h(z)}{2^{i-1}}\rfloor = \lfloor \frac{h(z^{\prime})}{2^{i-1}}\rfloor.
\end{equation}
\end{defn}

\begin{rem}\label{aadefoffunch}
The condition $(a)$ of Definition \ref{adefoffunch} is equivalent to the following condition $(b)$.\\
$(b)$ Suppose that $z_k,y_k \in \{0,1\}$ for $k = 0,1,...,n$ and 
\begin{equation}\label{conditionforh}
h(\sum\limits_{k = 0}^n {{z_k}} {2^k})=\sum\limits_{k = 0}^n {{y_k}} {2^k}.\nonumber
\end{equation}
Let $i$ be a natural number. Then for any $z_k^{\prime} \in \{0,1\}$ for $k=0,...,i-1$ there exist $y_k^{\prime} \in \{0,1\}$ for $k=0,...,i-2$ such that 

\begin{equation}\label{conditionforh2}
h(\sum\limits_{k = i}^n {{z_k}} {2^k} + \sum\limits_{k = 0}^{i-1} {{z_k^{\prime}}} {2^k}) = \sum\limits_{k = i-1}^n {{y_k}} {2^k} + \sum\limits_{k = 0}^{i-2} {{y_k^{\prime}}} {2^k}.\nonumber
\end{equation}
\end{rem}

The condition $(a)$ of Definition \ref{adefoffunch} is very abstract, so 
we present some examples of functions that satisfy condition $(a)$ of Definition \ref{adefoffunch} in Lemma \ref{lemmaforfloorzbyk} and Lemma \ref{lemmafor2overlog}.
\begin{Lem}\label{lemmaforfloorzbyk}
Let $h(z)=\lfloor \frac{z}{2k}\rfloor$ for some natural number $k$. Then 
$h(z)$ satisfies the condition of Definition \ref{adefoffunch}.
\end{Lem}
\begin{proof} 
We prove the contraposition of the condition of Definition \ref{adefoffunch}. We suppose that Equation (\ref{aaconditionforh2}) is false. Then there exist $u \in Z_{\geq 0}$    
 and a natural number $i$ such that 
\begin{equation}\label{falseofformer}
\lfloor \frac{\lfloor \frac{z}{2k}  \rfloor}{2^{i-1}}\rfloor = u < u+1 \leq  \lfloor \frac{\lfloor \frac{z^{\prime}}{2k}  \rfloor}{2^{i-1}}\rfloor.
\end{equation}
We prove that Equation (\ref{aaconditionforh}) is false.
From  the inequality in  (\ref{falseofformer}), we have 
\begin{equation}
\lfloor \frac{z}{2k}  \rfloor \leq u2^{i-1}+2^{i-1}-1 < (u+1)2^{i-1} \leq  \lfloor \frac{z^{\prime}}{2k}  \rfloor,\nonumber 
\end{equation}
and hence 
\begin{equation}\label{multplefy2k}
z \leq  2k(u2^{i-1}+2^{i-1}-1)+2k-1 < 2k(u+1)2^{i-1} \leq  z^{\prime}.
\end{equation}
From the inequality in  (\ref{multplefy2k}), we have 
\begin{equation}
\frac{z}{2^i} \leq  k(u+1)-\frac{1}{2^i} < k(u+1) \leq  \frac{z^{\prime}}{2^i}.\nonumber 
\end{equation}
Therefore we have
\begin{equation}
\lfloor \frac{z}{2^i}  \rfloor < k(u+1) \leq  \lfloor \frac{z^{\prime}}{2^i}  \rfloor.\nonumber 
\end{equation}
This shows that Equation (\ref{aaconditionforh}) is false. Therefore, we have completed the proof of this lemma.
\end{proof}

\begin{Lem}\label{lemmafor2overlog}
	Let $h(0) = h(1)=0$ and $h(z)=2^{\lfloor log_2z\rfloor-1}$ for $z \in Z_{\geq 0}$ such that $z \geq 2$. Then 
	$h(z)$ satisfies the condition of Definition \ref{adefoffunch}.
\end{Lem}
\begin{proof} 
	We prove the contraposition of the condition of Definition \ref{adefoffunch}.
	Suppose that Equation (\ref{aaconditionforh2}) is false. Then for a natural number $i$ 
	\begin{equation}\label{falseofflog1}
		\lfloor \frac{2^{\lfloor log_2z\rfloor-1}}{2^{i-1}}\rfloor  <   	\lfloor \frac{2^{\lfloor log_2z^{\prime}\rfloor-1}}{2^{i-1}}\rfloor.
	\end{equation}
	We prove that Equation (\ref{aaconditionforh}) is false. Let $z = 2^{n+m}$ and $z^{\prime}=2^{n^{\prime}+m^{\prime}}$ such that $n, n^{\prime} \in  Z_{\geq0}$ and $0 \leq m,m^{\prime} < 1$. Then 
	by the inequality in  (\ref{falseofflog1}), we have 
	\begin{equation}\label{falseofflog2}
	\lfloor 2^{n-i}  \rfloor < \lfloor 2^{n^{\prime}-i}  \rfloor.
	\end{equation}
	By the inequality in  (\ref{falseofflog2}), we have 
	\begin{align}
    n^{\prime} - i \geq 0\label{agrulog1} \\
    \text{and } \nonumber\\
    n+1 \leq n^{\prime}.\label{agrulog2}  
	\end{align}
	By the inequalities in (\ref{agrulog1}) and (\ref{agrulog2}), we have 

	\begin{equation}
	\lfloor \frac{z}{2^i}  \rfloor = 	\lfloor 2^{n+m-i}  \rfloor <  2^{n^{\prime}-i} \leq  \lfloor 2^{n^{\prime}+m^{\prime}-i}  \rfloor		= \lfloor \frac{z^{\prime}}{2^i}  \rfloor.\nonumber 
	\end{equation}
This shows that Equation (\ref{aaconditionforh}) is false. Therefore we have completed the proof.
\end{proof}

In the remainder of this subsection we assume that $h$ is the function that satisfies the condition $(a)$ in Definition \ref{adefoffunch}.
Our aim is to show that  the disjunctive sum of the chocolate bar  game with $CB(f,y,z)$ to the right of the bitter square and a single strip of chocolate bar  to the left have  $\mathcal{P}$-positions when $x\oplus y\oplus z=0$, so that the Grundy number of the chocolate bar $CB(f,y,z)$ is $x$ $ =y\oplus z$. 

We need Lemma \ref{fromNtoPforh} and Lemma \ref{fromPtoNforh} for this aim.
Lemma \ref{fromNtoPforh} implies that from a position $\{x,y,z\}$ of the disjunctive sum such that  $x \oplus y \oplus z \neq 0$ you always have a option that leads to a position for which the nim-sum of the coordinates is $0$.
Lemma \ref{fromPtoNforh} implies that from a position $\{x,y,z\}$ of the disjunctive sum such that  $x \oplus y \oplus z = 0$ any option leads to a position for which the nim-sum of the coordinates is not  $0$.\\
To prove Lemma \ref{fromNtoPforh} and Lemma \ref{fromPtoNforh}  we need some properties of the function $h$ that satisfies the condition $(a)$ in Definition \ref{adefoffunch}. These properties are proved in Lemma \ref{lemmaforfunctionh}, Lemma \ref{lemmaforfunctionh2} and Lemma \ref{lemmaforfunctionh3}. 

\begin{Lem}\label{lemmaforfunctionh}
Suppose that 
\begin{equation}\label{eqlemmaforh1}
h(\sum\limits_{k = 0}^n {{z_k}} {2^k}) \geq \sum\limits_{k = 0}^n {{y_k}} {2^k}. 
\end{equation}
Then, for any natural number $i$,
\begin{equation}\label{eqlemmaforh2}
h(\sum\limits_{k = i}^n {{z_k}} {2^k}) \geq \sum\limits_{k = i-1}^n {{y_k}} {2^k}.\nonumber
\end{equation}
\end{Lem}
\begin{proof} Let $h(\sum\limits_{k = 0}^n {{z_k}} {2^k}) = \sum\limits_{k = 0}^n {{u_k}} {2^k}$ for $u_k \in \{0,1\}$. Then, by the inequality in (\ref{eqlemmaforh1}), 
\begin{equation}\label{eqqq1}
\sum\limits_{k = 0}^n {{u_k}} {2^k} \geq \sum\limits_{k = 0}^n {{y_k}} {2^k}, \nonumber
\end{equation}
and hence
\begin{equation}\label{eqqq2}
\sum\limits_{k = i-1}^n {{u_k}} {2^k} \geq \sum\limits_{k = i-1}^n {{y_k}} {2^k}. 
\end{equation}
Let $z_k^{\prime} = 0$ for $k = 0,1,2,...,i-1$.
By the inequality in (\ref{eqqq2}), Definition \ref{adefoffunch} and Remark \ref{aadefoffunch}, there exist $y_k^{\prime}$ for $k=0,...,i-2$ such that 
\begin{equation}\label{eqlemmaforh3}
h(\sum\limits_{k = i}^n {{z_k}} {2^k}) = h(\sum\limits_{k = i}^n {{z_k}} {2^k} + \sum\limits_{k = 0}^{i-1} {{z_k^{\prime}}} {2^k}) = \sum\limits_{k = i-1}^n {{u_k}} {2^k} + \sum\limits_{k = 0}^{i-2} {{y_k^{\prime}}} {2^k} \geq \sum\limits_{k = i-1}^n {{u_k}} {2^k} \geq \sum\limits_{k = i-1}^n {{y_k}} {2^k}.\nonumber
\end{equation}
\end{proof}

\begin{Lem}\label{lemmaforfunctionh2}
Suppose that $p_k \in Z_{\geq0}$ for $k=0,1,...,n$.\\
$(i)$ Suppose that for some natural number $i$
\begin{equation}\label{eqlemmaforh21}
h(\sum\limits_{k = i}^n {{p_k}} {2^k}) < \sum\limits_{k = i-1}^n {{q_k}} {2^k}.
\end{equation}
Then 
\begin{equation}\label{eqlemmaforh22}
h(\sum\limits_{k = 0}^n {{p_k}} {2^k}) < \sum\limits_{k = i-1}^n {{q_k}} {2^k}.\nonumber
\end{equation}
$(ii)$ Suppose that 
\begin{equation}\label{eqlemmaforh21b}
h(\sum\limits_{k = 0}^n {{p_k}} {2^k}) \geq \sum\limits_{k = i-1}^n {{q_k}} {2^k}.\nonumber
\end{equation}
Then 
\begin{equation}\label{eqlemmaforh22b}
h(\sum\limits_{k = i}^n {{p_k}} {2^k}) \geq \sum\limits_{k = i-1}^n {{q_k}} {2^k}.
\nonumber
\end{equation}
\end{Lem}
\begin{proof}
We prove $(i)$. Let $h(\sum\limits_{k = 0}^n {{p_k}} {2^k})= \sum\limits_{k = 0}^n {{r_k}} {2^k}$ for $r_k \in \{0,1\}$. Then, by Remark \ref{aadefoffunch}, there exist $r_k^{\prime}$ for $k=0,1,2,...,i-2$ such that 
$h(\sum\limits_{k = i}^n {{p_k}} {2^k} + \sum\limits_{k = 0}^{i-1} {0 \times } {2^k}) $ $= \sum\limits_{k = i-1}^n {{r_k}} {2^k}+\sum\limits_{k = 0}^{i-2} {{r_k^{\prime}}} {2^k}$.
Then, by the inequality in (\ref{eqlemmaforh21}), we have
\begin{equation}\label{eqlemmaforh23}
\sum\limits_{k = i-1}^{n} {{q_k}} {2^k} > \sum\limits_{k = i-1}^n {{r_k}} {2^k}+\sum\limits_{k = 0}^{i-2} {{r_k^{\prime}}} {2^k},\nonumber
\end{equation}
and hence we have the inequality in (\ref{firstcasefor}) or Relation (\ref{secondcasefor}).

\begin{equation}\label{firstcasefor}
q_n > r_n.
\end{equation}
There exists $j \in Z_{\geq0}$ such that 
\begin{equation}\label{secondcasefor}
i-1 \leq j \leq n, q_k = r_k \text{ for } k = j+1,j+2,...,n \text{ and } q_j > r_j. 
\end{equation}

Then, by the inequality in (\ref{firstcasefor}) or Relation (\ref{secondcasefor}), 
\begin{equation}\label{eqlemmaforh24}
\sum\limits_{k = i-1}^{n} {{q_k}} {2^k} > \sum\limits_{k = i-1}^n {{r_k}} {2^k}+\sum\limits_{k = 0}^{i-2} {{r_k}} {2^k} = \sum\limits_{k = 0}^n {{r_k}} {2^k}=h(\sum\limits_{k = 0}^n {{p_k}} {2^k}).\nonumber
\end{equation}
We prove $(ii)$. This is the contraposition of the proposition in $(i)$ of this lemma.
\end{proof}

\begin{Lem}\label{lemmaforfunctionh3}
Suppose that 
	\begin{equation}\label{eqlemmaforh212}
	h(\sum\limits_{k = i}^n {{p_k}} {2^k}) < \sum\limits_{k = i-1}^n {{q_k}} {2^k}+2^{i-1}.
	\end{equation}
	Then 
	\begin{equation}\label{eqlemmaforh222}
	h(\sum\limits_{k = i-1}^n {{p_k}} {2^k}) < \sum\limits_{k = i-1}^n {{q_k}} {2^k}+2^{i-1}.\nonumber
	\end{equation}
\end{Lem}
\begin{proof}
Let $\sum\limits_{k = i-1}^{n+1} {{q^{\prime}_k}} {2^k} = \sum\limits_{k = i-1}^n {{q_k}} {2^k}+2^{i-1}$ and $p_{n+1}=0$. Then, by the inequality in (\ref{eqlemmaforh212}), 
we have $h(\sum\limits_{k = i}^{n+1} {{p_k}} {2^k}) < \sum\limits_{k = i-1}^{n+1} {{q^{\prime}_k}} {2^k}$, and statement 
$(i)$ of Lemma \ref{lemmaforfunctionh2} implies 
$	h(\sum\limits_{k = i-1}^n {{p_k}} {2^k}) \leq h(\sum\limits_{k = 0}^{n+1} {{p_k}} {2^k}) < \sum\limits_{k = i-1}^{n+1} {{q^{\prime}_k}} {2^k}$ $=\sum\limits_{k = i-1}^n {{q_k}} {2^k}+2^{i-1}$. Therefore we have completed the proof of this lemma.
\end{proof}

If the nim-sum of the cooridinates of a position is not 0, then 
by Definition \ref{moveforthreecordinates} and the following Lemma \ref{fromNtoPforh},
there is always an option that leads to a position whose nim-sum is 0.

\begin{Lem}\label{fromNtoPforh}
Suppose that
 $x \oplus y \oplus z \neq 0$ and 
\begin{equation}\label{inequalityyz}
y \leq h(z).
\end{equation}
 
Then at least one of the following statements is true.\\
$(1)$ $u\oplus y\oplus z= 0$  for some $u\in Z_{\geq 0}$ such that  $u<x$.\\
$(2)$ $x\oplus v\oplus z= 0$  for some $v\in Z_{\geq 0}$ such that  $v<y$.\\
$(3)$ $x\oplus y\oplus w= 0$ for some $w\in Z_{\geq 0}$ such that  $w<z$ and $y \leq h(w)$.\\
$(4)$ $x\oplus v\oplus w^{\prime}= 0$   for some $v,w^{\prime} \in Z_{\geq 0}$ such that  $v<y,w^{\prime} <z$ and $v=h(w^{\prime})$.\\
\end{Lem}
\begin{proof}
Let $x= \sum\limits_{k = 0}^n {{x_k}} {2^k}$, $y= \sum\limits_{k = 0}^n {{y_k}} {2^k}$ and $z= \sum\limits_{k = 0}^n {{z_k}} {2^k}$. If $n=0$, then this lemma is obvious. We assume that $n \geq 1$.
Suppose that there exists a non-negative integer $s$ such that $x_i + y_i + z_i = 0 \ (mod \ 2)$ for $i = n,n-1,...,n-s$ and  
\begin{equation}\label{xyznminush}
x_{n-s-1} + y_{n-s-1} + z_{n-s-1} \neq 0 \ (mod \ 2).
\end{equation}
\underline{Case $(i)$}  Suppose that $x_{n-s-1}=1$. Then, we define $u = \sum^{n}_{i=1}u_i2^i$ by $u_i=x_i$ for $i=n,n-1,..., n-s $,  $u_{n-s-1}=0<x_{n-s-1}$ and $u_i=y_i+z_i$ $( mod \ 2 )$ for $i=n-s-2,n-s-3,...,0$. Then we have $u \oplus y \oplus z = 0$ and $u<x$. Therefore, we have statement $(1)$ of this lemma.\\
\underline{Case $(ii)$} Suppose that $y_{n-s-1}=1$. Then, by the method that is similar to the one  used in $(i)$, we prove that $x \oplus v \oplus z = 0$ for some $v \in Z_{\geq 0}$ such that $v<y$. Therefore we have statement $(2)$ of this lemma.\\
\underline{Case $(iii)$} We suppose that 
\begin{equation}\label{xyztonmink}
 z_{n-s-1} =1.
\end{equation}
For $i = n,n-1,...,n-s$, let
\begin{equation}\label{fromntonminask}
w_i = z_i.
\end{equation}
Let $w_i = x_i + y_i \ (mod \ 2)$ for $i=n-s-1,...,0$. 
By the inequality in (\ref{xyznminush}) and Equation (\ref{xyztonmink}), we have 
 $w_{n-s-1}$ $=x_{n-s-1}+y_{n-s-1} = 0$ $(mod \ 2)$, and hence
\begin{equation}\label{newinnauai1}
w_{n-s-1}=0 < 1 = z_{n-s-1}.
\end{equation}
There are two subcases here.\\
\underline{Subcase $(iii.1)$} If $y \leq h(w)$, then we have statement $(3)$ of this lemma. \\
\underline{Subcase $(iii.2)$ } Next we suppose that 
\begin{equation}\label{ybiggernthanhw}
y > h(w).
\end{equation}
By the inequality in (\ref{inequalityyz}), we have $\sum\limits_{k = 0}^n {{y_k}} {2^k}
\leq h(\sum\limits_{k = 0}^n {{z_k}} {2^k})$, and hence by  Lemma \ref{lemmaforfunctionh} and (\ref{fromntonminask}) 
\begin{equation}\label{untilies1}
\sum\limits_{k = n-s-1}^n {{y_k}} {2^k}
\leq h(\sum\limits_{k = n-s}^n {{z_k}} {2^k})=h(\sum\limits_{k = n-s}^n {{w_k}} {2^k}) \leq h(w).
\end{equation}
By  the inequalities in (\ref{untilies1}) and  (\ref{ybiggernthanhw}), there exists a natural number $j$ such that 

\begin{align}
\sum\limits_{k = n-j}^n {{y_k}} {2^k}
\leq h(\sum\limits_{k = 0}^n {{w_k}} {2^k}) =h(w)\label{e1quationse31}\\
\text{and}
\sum\limits_{k = n-j-1}^n {{y_k}} {2^k}
> h(\sum\limits_{k = 0}^n {{w_k}} {2^k})=h(w) .\label{e1quationse32}
\end{align}     
By the inequalities in (\ref{untilies1}) and  (\ref{e1quationse32}),
\begin{equation}\label{star1ineq}
n-j-1 < n-s-1.
\end{equation}
By the inequality in (\ref{e1quationse31}) and $(ii)$ of Lemma \ref{lemmaforfunctionh2}, 
\begin{equation}\label{e1quationse31b}
\sum\limits_{k = n-j}^n {{y_k}} {2^k}
\leq h(\sum\limits_{k = n-j+1}^n {{w_k}} {2^k})\leq h(\sum\limits_{k = n-j}^n {{w_k}} {2^k}).
\end{equation}
By the inequality in (\ref{e1quationse32}), 
\begin{equation}\label{e1quationse32b}
\sum\limits_{k = n-j-1}^n {{y_k}} {2^k}
> h(\sum\limits_{k = n-j}^n {{w_k}} {2^k}).
\end{equation}
By the inequalities in (\ref{e1quationse31b}) and  (\ref{e1quationse32b}), we have 
\begin{equation}\label{e1quationse31and31b}
\sum\limits_{k = n-j}^n {{y_k}} {2^k}
\leq h(\sum\limits_{k = n-j}^n {{w_k}} {2^k}) < \sum\limits_{k = n-j-1}^n {{y_k}} {2^k}.
\end{equation}
We construct $v$ and $w^{\prime}$ by assigning values to $v_i$ and $w_i^{\prime}$ for $i = n,n-1,n-2,...,0$. 

First, for $i = n,n-1,...,n-j$, let 
\begin{equation}\label{constructvandwprime}
w_i^{\prime} = w_i \text{ \ and } v_i = y_i,
\end{equation}
and let 
\begin{equation}\label{star2ineq}
v_{n-j-1} = 0 < 1 = y_{n-j-1} 
\end{equation}
and 
\begin{equation}
w_{n-j-1}^{\prime}=x_{n-j-1} + v_{n-j-1}. \nonumber
\end{equation}
Since $v_{n-j-1} = 0$ and $y_{n-j-1}=1$, by the inequality in (\ref{e1quationse31and31b})
\begin{equation}\label{constru1}
\sum\limits_{k = n-j-1}^n {{v_k}} {2^k} 
\leq h(\sum\limits_{k = n-j}^n {{w_k^{\prime}}} {2^k}) < \sum\limits_{k = n-j-1}^n {{v_k}} {2^k} + 2^{n-j-1}.
\end{equation}
By the inequality in (\ref{constru1}) and Lemma \ref{lemmaforfunctionh3}, we have 

\begin{equation}\label{constru1b}
\sum\limits_{k = n-j-1}^n {{v_k}} {2^k} 
\leq h(\sum\limits_{k = n-j-1}^n {{w_k^{\prime}}} {2^k}) < \sum\limits_{k = n-j-1}^n {{v_k}} {2^k} + 2^{n-j-1}.
\end{equation}
Next we prove the inequality in (\ref{construct1}) for any $t=n-j-1, n-j-2,...,2,1,0$ recursively. 
\begin{equation}\label{construct1}
\sum\limits_{k = t}^n {{v_k}} {2^k}
\leq h(\sum\limits_{k = t}^n {{w_k^{\prime}}} {2^k}) < \sum\limits_{k = t}^n {{v_k}} {2^k} + 2^t.
\end{equation}
By the inequality in (\ref{constru1b}), we have the inequality in (\ref{construct1}) for $t = n-j-1$.
We suppose the inequality in (\ref{construct1}) for some natural number $t$ such that $t \leq n-j-1$. Then we have the inequality in (\ref{construct2}) or the inequality in (\ref{construct4}). Our aim is to prove (\ref{construct3b}) and (\ref{construct5b}) by using these inequalities.

If
\begin{equation}\label{construct2}
\sum\limits_{k = t}^n {{v_k}} {2^k}+2^{t-1}
\leq h(\sum\limits_{k = t}^n {{w_k^{\prime}}} {2^k}) < \sum\limits_{k = t}^n {{v_k}} {2^k}+2^{t},
\end{equation}
then let $v_{t-1}=1$ and $w_{t-1}^{\prime}=x_{t-1}+ v_{t-1}$ $(mod \ 2)$. Since $v_{t-1}=1$, by the inequality in (\ref{construct2})  we have 
\begin{equation}\label{construct3}
\sum\limits_{k = t-1}^n {{v_k}} {2^k}
\leq h(\sum\limits_{k = t}^n {{w_k^{\prime}}} {2^k}) < \sum\limits_{k = t-1}^n {{v_k}} {2^k}+2^{t-1}.
\end{equation}
Note that $v_{t-1}2^{t-1} + 2^{t-1} = 2^t$. 
By Lemma \ref{lemmaforfunctionh3} and the inequality in (\ref{construct3}), 
\begin{equation}\label{construct3b}
\sum\limits_{k = t-1}^n {{v_k}} {2^k}
\leq h(\sum\limits_{k = t-1}^n {{w_k^{\prime}}} {2^k}) < \sum\limits_{k = t-1}^n {{v_k}} {2^k}+2^{t-1}.
\end{equation}

If
\begin{equation}\label{construct4}
\sum\limits_{k = t}^n {{v_k}} {2^k}+2^{t-1}
> h(\sum\limits_{k = t}^n {{w_k^{\prime}}} {2^k}),
\end{equation}
then let $v_{t-1}=0$ and $w_{t-1}^{\prime}=x_{t-1}+ v_{t-1}$ $(mod \ 2)$.\\
Since $v_{t-1}=0$, the inequalities in (\ref{construct4}) and  (\ref{construct1}) give 
\begin{equation}\label{construct5}
\sum\limits_{k = t-1}^n {{v_k}} {2^k}
\leq h(\sum\limits_{k = t}^n {{w_k^{\prime}}} {2^k}) < \sum\limits_{k = t-1}^n {{v_k}} {2^k}+2^{t-1}.
\end{equation}
Then, by the inequality in (\ref{construct5}) and Lemma \ref{lemmaforfunctionh3},  we have 
\begin{equation}\label{construct5b}
\sum\limits_{k = t-1}^n {{v_k}} {2^k}
\leq h(\sum\limits_{k = t-1}^n {{w_k^{\prime}}} {2^k}) < \sum\limits_{k = t-1}^n {{v_k}} {2^k}+2^{t-1}.
\end{equation}

In this way we get the inequality in (\ref{construct3b}) or the inequality in (\ref{construct5b})  by the inequality in (\ref{construct1}). Note that  the inequality in (\ref{construct3b}) and  the inequality in (\ref{construct5b})  are the same inequality.
By continuing this process we have 
\begin{equation}
\sum\limits_{k = 0}^n {{v_k}} {2^k}
\leq h(\sum\limits_{k = 0}^n {{w_k^{\prime}}} {2^k}) < \sum\limits_{k = 0}^n {{v_k}} {2^k}+2^{0}.\nonumber
\end{equation}
Therefore, we have
$\sum\limits_{k = 0}^n {{v_k}} {2^k}
= h(\sum\limits_{k = 0}^n {{w_k^{\prime}}} {2^k}).$ By iequalities (\ref{newinnauai1}), (\ref{star1ineq}), (\ref{star2ineq}) and Equation (\ref{constructvandwprime}), we have
$v < y$ and $w^{\prime} < z$. Therefore, we have statement $(4)$ of this lemma.
\end{proof}
If the nim-sum of the cooridinates of a position is 0, then 
by Definition \ref{moveforthreecordinates} and the following Lemma \ref{fromPtoNforh},
any option from this position leads to a position whose nim-sum is not 0.

\begin{Lem}\label{fromPtoNforh}
If $x \oplus y \oplus z = 0$ and $y \le h(z)$
, then the following hold:\\
$(i)$  $u \oplus y \oplus z \ne 0$ for any $u \in {Z_{ \ge 0}}$ such that $u<x$.\\
$(ii)$  $x \oplus v \oplus z \ne 0$ for any $v \in {Z_{ \ge 0}}$  such that $v<y$.\\
$(iii)$  $x \oplus y \oplus w \ne 0$ for any $w \in {Z_{ \ge 0}}$   such that  $w<z$.\\
$(iv)$   $x \oplus v \oplus w \ne 0$ for any $v,w \in {Z_{ \ge 0}}$  such that $v<y,w<z$ and $v= h(w)$.
\end{Lem}
\begin{proof}
Statements $(i)$,$(ii)$ and $(iii)$ of this lemma follow directly from Definition \ref{definitionfonimsum11} (the definition of nim-sum). \\
We now prove statement $(iv)$. We suppose that $v= h(w)$ for some $w\in Z_{\geq 0}$ such that $v<y,w<z$. 
We also suppose that 
\begin{equation}\label{construct6a}
w_i=z_i \text{ for } i = n,n-1,n-2,...,j \text{ and } w_{j-1}<z_{j-1}.
\end{equation}
By $y \le h(z)$, we have $h(\sum\limits_{k = 0}^n {{z_k}} {2^k}) \geq  \sum\limits_{k = 0}^n {{y_k}} {2^k}$. Hence, by Lemma \ref{lemmaforfunctionh}, we have 

\begin{equation}\label{construct6}
h(\sum\limits_{k = j}^n {{z_k}} {2^k}) \geq  \sum\limits_{k = j-1}^n {{y_k}} {2^k}.
\end{equation}
Since $v= h(w)$ and $v < y$, Relation (\ref{construct6a}) gives
\begin{equation}\label{construct7}
h(\sum\limits_{k = j}^n {{z_k}} {2^k}) = h(\sum\limits_{k = j}^n {{w_k}} {2^k}) \leq h(w) = v =  \sum\limits_{k = 0}^n {{v_k}} {2^k} < \sum\limits_{k = 0}^n {{y_k}} {2^k}.
\end{equation}
By the inequalities in (\ref{construct6}) and (\ref{construct7}), we have 
\begin{equation}
\sum\limits_{k = j-1}^n {{y_k}} {2^k} \leq \sum\limits_{k = 0}^n {{v_k}} {2^k} < \sum\limits_{k = 0}^n {{y_k}} {2^k}. \nonumber 
\end{equation}
Hence, for $k = n,n-1,n-2,...,j-1$,
\begin{equation}\label{conditionyv1}
v_{k}=y_{k}. 
\end{equation}
Since $x \oplus y \oplus z = 0$, we have 
\begin{equation}\label{conditionyv2}
x_{j-1}+y_{j-1}+z_{j-1} = 0 \ (mod \ 2).
\end{equation}
By Relation (\ref{construct6a}), Equation (\ref{conditionyv1}) and  Equation (\ref{conditionyv2}),
 we have $x_{j-1}+v_{j-1}+w_{j-1} \neq 0 \ (mod \ 2)$, and hence 
 $x \oplus v \oplus w \neq 0$.
\end{proof}

\begin{defn}\label{defofABk}
Let $A_{h}=\{\{x,y,z\}:x,y,z\in Z_{\geq 0},y \leq h(z) $ and $x\oplus y \oplus z=0\}$ and  $B_{h}=\{\{x,y,z\}:x,y,z\in Z_{\geq 0},y \leq h(z)$ and $x\oplus y \oplus z\neq 0\}$.
\end{defn}

\begin{Lem}\label{AtoBandBtoA}
	Let $A_h$ and $B_h$ be the sets defined in Definition \ref{defofABk}. Then the following hold:\\
$(i)$ If we start with a position in $A_h$, then any option (move) leads to a position in $B_h$.\\
$(ii)$ If we start with a position in $B_h$, then there is at least one option (move) that leads to a position in $A_h$.\\
\end{Lem}
\begin{proof}
Since  $move_h (\{x, y, z\})$ that is defined in Definition \ref{moveforthreecordinates} contains all the positions that can be reached from the position $\{x, y, z\}$ in one step, we have statements $(i)$ and $(ii)$ by Lemma \ref{fromPtoNforh} and Lemma \ref{fromNtoPforh} respectively.
\end{proof}

\begin{thm}\label{thmforPNposition}
Let $A_h$ and $B_h$ be the sets defined in Definition \ref{defofABk}.
$A_h$ is the set of $\mathcal{P}$-positions and $B_h$   is the set of $\mathcal{N}$-positions of the disjunctive sum of the chocolate bar  game with $CB(h,y,z)$ to the right of the bitter square and a single strip of chocolate bar  to the left.
\end{thm}
\begin{proof}
If we start the game from a position $\{x,y,z\}\in A_{h}$, then Lemma \ref{AtoBandBtoA} indicates that any option we take leads to a position $\{p,q,r\}$ in $B_h$. From this position  $\{p,q,r\}$, Lemma \ref{AtoBandBtoA}    
 implies that our opponent can choose a proper option that leads to a position in $A_h$. Note that any option reduces some of the numbers in the coordinates. In this way, our opponent can always reach a position in $A_h$, and will finally win by reaching $\{0,0,0\}\in A_{h}$. Note that position $\{0,0,0\}$ represent  the bitter square itself, and we cannot eat this part.
Therefore $A_h$ is the set of $\mathcal{P}$-positions.\\
If we start the game from a position $\{x,y,z\}\in B_{h}$, then Lemma \ref{AtoBandBtoA} means that we can choose a proper option that leads to a position  $\{p,q,r\}$ in $A_h$. 
From $\{p,q,r\}$, Lemma \ref{AtoBandBtoA} implies that  any option taken by our opponent leads to a position in $B_h$. In this way we win the game by reaching $\{0, 0, 0\}$.
Note that our opponent cannot eat the bitter part. Therefore $B_h$ is the set of $\mathcal{N}$-position
\end{proof}

\begin{thm}\label{theoregrundyforf}
Let $h$ be the function that satisfies the condition $(a)$ in Definition \ref{adefoffunch}. Then the Grundy number of $CB(h,y,z)$ is $ y \oplus z$.
\end{thm}
\begin{proof}
By Theorem  \ref{thmforPNposition},  a position $\{x,y,z\}$ of the sum of the chocolate bars  is a $\mathcal{P}$-position  when $x\oplus y\oplus z=0$. Therefore,  Theorem \ref{theoremofsumg} implies that the Grundy number of the chocolate bar to the right is $x$ $ = y\oplus z$.
\end{proof}
By Theorem \ref{theoregrundyforf}, the condition of Definition \ref{adefoffunch} is a sufficient condition for the chocolate bar $CB(h,y,z)$ to have the Grundy number $G(\{y,z\}) = y \oplus z$. Lemma \ref{lemmaforfloorzbyk}, Lemma \ref{lemmafor2overlog} and 
Theorem \ref{theoregrundyforf} imply that chocolate bar games in Figure \ref{referee1p}, Figure \ref{referee2p}, Figure \ref{referee3p} and Figure \ref{yzk4h2graph} have the Grundy number $G(\{y,z\}) = y \oplus z$.

In the next subsection we prove that the condition of Definition \ref{adefoffunch} is a necessary condition for the chocolate bar $CB(h,y,z)$ to have the Grundy number $G(\{y,z\}) = y \oplus z$.

\subsection{A Necessary Condition for a Chocolate Bar to have the Grundy Number $y \oplus z$}
In Subsection \ref{subsectionforsufficondi}, we proved that the Grundy number $G(\{y,z\}) = y \oplus z$ for $CB(h,y,z)$  when 
the function $h$ satisfies the condition $(a)$ in Definition \ref{adefoffunch}.

In this subsection, we prove that the condition $(a)$ in Definition \ref{adefoffunch} is a necessary condition for $f$ to  have the Grundy number $ y \oplus z$ for the chocolate bar   
 $CB(f,y,z)$. 
\begin{defn}\label{adefoff}
Let $f$ be a monotonically increasing function of $Z_{\geq0}$ into $Z_{\geq0}$ that satisfies the following condition $(a)$.\\
$(a)$ Suppose that $G(\{y,z\})$ is the Grundy number of the chocolate bar $CB(f,y,z)$. Then,  
\begin{equation}
G(\{y,z\}) = y \oplus z.\nonumber
\end{equation}
\end{defn}

Throughout this subsection we assume that the function $f$ satisfies the condition $(a)$ of Definition \ref{adefoff}, and 
we prove that this function $f$ satisfies the condition $(a)$ of Definition \ref{adefoffunch} using the following 
Lemma \ref{acomparegrundies}, Lemma \ref{agrundyequallemma} and Lemma \ref{asnlemma}.

\begin{Lem}\label{acomparegrundies}
Let $y,z,y^{\prime} \in Z_{\geq0}$ such that $y=f(z)$, $y^{\prime} \leq f(z+1)$ and $y < y^{\prime}$. Then, 
$G(\{y,z+1\}) < G(\{y^{\prime},z+1\})$.
\end{Lem}
\begin{proof}
Since $y=f(z)$, we have $f(w) \leq y < y^{\prime}$ for $w \leq z$. Therefore, 
\begin{align*}
&move_f(\{y^{\prime},z+1\}) \\
&= \{\{v,z+1 \}:v<y^{\prime} \} \cup  \{ \{\min(y^{\prime}, f(w)),w \}:w<z+1 \}\\
&= \{\{v,z+1 \}:v<y^{\prime} \} \cup  \{ \{f(w),w \}:w<z+1 \}\\
&=\{\{v,z+1\}: y \leq v < y^{\prime} \} \cup  \{\{v,z+1 \}:v<y \} \cup  \{ \{f(w),w \}:w<z+1 \}\\
&=\{\{v,z+1\}: y \leq v < y^{\prime} \} \cup  \{\{v,z+1 \}:v<y \} \cup  \{ \{\min(y, f(w)),w \}:w<z+1 \}\\
&=\{\{v,z+1\}: y \leq v < y^{\prime} \}  \cup move_f(\{y,z+1\}), \ where \ v,w \in Z_{\ge 0}.
\end{align*}

Therefore, 
\begin{align}
&G(\{y^{\prime},z+1\}) \nonumber \\
&= \textit{mex}(\{G(\{v,z+1\}):y \leq v < y^{\prime} \}  \cup \{G(\{a,b\}) \nonumber\\
&:\{a,b\} \in move_f(\{y,z+1\})\} \geq G(\{y,z+1\}). \label{grundylqgrundy}
\end{align}

Since $\{y,z+1\} \in move_f(\{y^{\prime},z+1\})$, $G(\{y^{\prime},z+1\}) \neq G(\{y,z+1\})$. Therefore, (\ref{grundylqgrundy}) implies  $G(\{y,z+1\}) < G(\{y^{\prime},z+1\})$.
\end{proof}

\begin{Lem}\label{agrundyequallemma}
For any $y,z \in Z_{\geq0}$ such that $y \leq f(z)$, we have
\begin{align*}
\{G(\{ \min(y,f(w)),w\}):w<z\} = \{y \oplus w:w < z\}.
\end{align*}
\end{Lem}
\begin{proof}
Let $w \in Z_{\geq0}$ such that $w< z$, and let
\begin{equation}\label{2nbiggerywz}
n = \lfloor log_2 \max(y,z) \rfloor +1.
\end{equation}
Then, Equation (\ref{2nbiggerywz}) implies that 
$y \oplus w < y \oplus (z+2^n) = G(\{y,z+2^n\})$.
By the definition of Grundy number, there exist $a,b \in Z_{\geq0}$ such that
$\{a,b\} \in move_f(\{y,z+2^n\})$ and $G(\{a,b\})=y \oplus w$.\\
By the definition of $move_f$, we have the following 
Equation (\ref{bydefofmove1}) or Equation (\ref{bydefofmove2}).
\begin{equation}\label{bydefofmove1}
G(\{ \min(y,f(w^{\prime})), w^{\prime} \}) = y \oplus w
\end{equation}
for $w^{\prime} \in Z_{\geq0}$ with $w^{\prime} < z+2^n$.

\begin{equation}\label{bydefofmove2}
y^{\prime} \oplus( z+2^n) = G(\{y^{\prime}, z+2^n\}) = y \oplus w
\end{equation}
for $y^{\prime} \in Z_{\geq0}$ with $y^{\prime} < y$.
Equation (\ref{bydefofmove2}) contradicts Equation (\ref{2nbiggerywz}), and hence we have
Equation (\ref{bydefofmove1}).
If 
\begin{equation}\label{ifsentence}
w^{\prime} \geq z, 
\end{equation}
then $f(w^{\prime}) \geq f(z) \geq y$. Hence,
\begin{equation}\label{grundyyszp}
G(\{ \min(y,f(w^{\prime})), w^{\prime} \}) =  G(\{y,w^{\prime} \}) = y \oplus w^{\prime}.
\end{equation}
By Equations (\ref{bydefofmove1}) and (\ref{grundyyszp}), we have 
\begin{equation}\label{leadcontra1}
y \oplus w = y \oplus w^{\prime}.
\end{equation}
Since $w^{\prime} < w$, Equation (\ref{leadcontra1}) leads to a contradiction. Therefore, the inequality in (\ref{ifsentence}) is false, and we have
$w^{\prime} < z$. Hence, Equation (\ref{bydefofmove1}) implies that $\{y \oplus w:w < z\} \subset \{G(\{ \min(y,f(w)),w\}):w<z\}$. The number of elements in $\{y \oplus w:w < z\}$ is the same as the number of elements in  $ \{G(\{ \min(y,f(w)),w\}):w<z\}$, and hence we have $\{G(\{ \min(y,f(w)),w\}):w<z\} = \{y \oplus w:w < z\}$.
\end{proof}

\begin{Lem}\label{asnlemma}
Let 
\begin{equation}
a= d \times 2^{i+1} + d_i 2^i + e-1 \nonumber
\end{equation}
for $d,e,i \in Z_{\geq 0} $, $d_i \in \{0,1\}$, $e < 2^i$ and $0 < d_i 2^i + e$.\\
If $c \times 2^{i+1} \leq f(a) < c \times 2^{i+1} + 2^i$ for $c \in  Z_{\geq 0}$, then 
$f(a+1) < c \times 2^{i+1} + 2^i$.
\end{Lem}
\begin{proof}
Let
\begin{equation}\label{asnlemma1}
f(a) = c \times 2^{i+1} + t
\end{equation}
for $0 \leq t < 2^i$.
We suppose that 
\begin{equation}\label{aleadtocont}
f(a+1) \geq c \times 2^{i+1} + 2^i, 
\end{equation}
and we show that this leads to a contradiction.\\
\underline{Case $(i)$} If $d_i = 1$, then \\
$G(\{c \times 2^{i+1}+2^i,a+1\}) = (c \times 2^{i+1} + 2^i) \oplus (d \times 2^{i+1} + d_i 2^i + e)$ \\  
\begin{equation}\label{contradlemma35}
= (c \oplus d) 2^{i+1} + e < (c \oplus d) 2^{i+1} + d_i 2^i + (t\oplus e) = G(\{c \times 2^{i+1} + t, a +1\}).
\end{equation}
By Equation (\ref{asnlemma1}), the inequality in (\ref{aleadtocont}) and Lemma \ref{acomparegrundies}, we have 
$G(\{c \times 2^{i+1} + t, a +1\})< G(\{c \times 2^{i+1}+2^i,a+1\})$, which contradicts Equation (\ref{contradlemma35}).\\
\underline{Case $(ii)$} If $d_i = 0$, then 
$G(\{c \times 2^{i+1}+2^i,a+1\}) = (c \times 2^{i+1} + 2^i) \oplus (d \times 2^{i+1}  + e)$ \\ $ = (c \oplus d) 2^{i+1} + 2^i + e > (c \oplus d) 2^{i+1} + 2^i $. Note that $e > 0$, since $d_i2^i + e > 0$ and $d_i = 0$.

Therefore, by the definition of Grundy number we have 
\begin{equation}\label{asnlemma2}
(c \oplus d) 2^{i+1} + 2^i \in \{G(\{p,q\}):\{p,q\} \in move_f(\{c \times 2^{i+1}+2^i,a+1\})\}.
\end{equation}

$\{G(\{p,q\}):\{p,q\} \in move_f(\{c \times 2^{i+1}+2^i,a+1\})\}$
\begin{align}
= \{G(\{v,d \times 2^{i+1}+e\}):v = 0,1,2,...,c \times 2^{i+1}+2^i-1\} \nonumber\\
\cup \{G(\{ \min(c \times 2^{i+1}+2^i,f(w)),w \}):w = 0,1,2,...,d \times 2^{i+1}+e-1 \}. \label{atransgrundy}
\end{align}
Note that $a=d \times 2^{i+1}+e-1$.

For $w \leq a$, we have $f(w) \leq f(a) = c\times 2^{i+1} + t$. Hence 

(\ref{atransgrundy})
\begin{align}
 = \{G(\{v,d \times 2^{i+1}+e\}):v = 0,1,2,...,c \times 2^{i+1}+2^i-1\} \nonumber\\
\cup \{G(\{ \min(c \times 2^{i+1}+t,f(w)),w \}):w = 0,1,2,...,d \times 2^{i+1}+e-1 \}.\label{lastofequation}
\end{align}
Since $c \times 2^{i+1}+t=f(a)$ $\leq f(a+1)=f(d \times 2^{i+1}+e)$, Lemma \ref{agrundyequallemma} implies that 
$\{G(\{ \min(c \times 2^{i+1}+t,f(w)),w \}):w = 0,1,2,...,d \times 2^{i+1}+e-1 \} = \{(c \times 2^{i+1}+t) \oplus w: w = 0,1,2,...,d \times 2^{i+1}+e-1 \}$. Therefore, by  Definition \ref{adefoff} 

(\ref{lastofequation})
\begin{align}
 = \{v \oplus (d \times 2^{i+1}+e): v = 0,1,2,...,c \times 2^{i+1}+2^i-1\} \nonumber\\
\cup \{(c \times 2^{i+1}+t) \oplus w: w = 0,1,2,...,d \times 2^{i+1}+e-1 \} \nonumber
\end{align}
\begin{align}
= \{(c \times 2^{i+1}+k)\oplus(d \times 2^{i+1}+e):k = 0,1,2,...,2^i-1\}\label{agrundyse1} \\
\cup \{k \oplus (d \times 2^{i+1}+e):k = 0,1,2,...,c \times 2^{i+1}-1\}\label{agrundyse2} \\
\cup \{ (c \times 2^{i+1}+t) \oplus (d \times 2^{i+1}+k):k = 0,1,2,...,e-1\}\label{agrundyse3} \\
\cup \{  (c \times 2^{i+1}+t) \oplus k:k = 0,1,2,...,d \times 2^{i+1}-1 \}.\label{agrundyse4}
\end{align}
Then we have the following statements $(i)$, $(ii)$, $(iii)$ and $(iv)$.\\
$(i)$ All the numbers in Set (\ref{agrundyse1}) are of the type $(c\oplus d)2^{i+1} + (k \oplus e)$, and hence this set does not contains $(c \oplus d) 2^{i+1} + 2^i$. Note that $k,e < 2^i$.\\
$(ii)$ The coefficients of $2^{i+1}$ of the numbers in Set (\ref{agrundyse2}) are not $c\oplus d$, and hence this set does not contains $(c \oplus d) 2^{i+1} + 2^i$.\\
$(iii)$ All the numbers in Set (\ref{agrundyse3}) are of the type $(c\oplus d)2^{i+1} + (t \oplus k)$, and hence this set does not contains $(c \oplus d) 2^{i+1} + 2^i$.\\ Note that $k \leq e-1 < 2^i$ and $t < 2^i$.\\
$(iv)$ The coefficients of $2^{i+1}$ of the numbers in Set (\ref{agrundyse4}) are not $c\oplus d$, and hence this set does not contains $(c \oplus d) 2^{i+1} + 2^i$.\\
Statements $(i)$, $(ii)$, $(iii)$ and $(iv)$ contradict Relation (\ref{asnlemma2}). Therefore we conclude that the inequality in (\ref{aleadtocont}) is false.
\end{proof}

\begin{thm}\label{necessarycondiforf}
Suppose that the function $f$ satisfies the condition $(a)$ in Definition \ref{adefoff}.\\
Then function $f$ satisfies the condition $(a)$ in Definition \ref{adefoffunch}.
\end{thm}
\begin{proof}
If $f$ does not satisfy the condition in Definition \ref{adefoffunch}, then there exist $z,z^{\prime} \in Z_{\geq 0}$ and a natural number $j$ such that 
$z < z^{\prime}$, 
\begin{equation}\label{donotsatisfy1}
\lfloor \frac{z}{2^j}\rfloor = \lfloor \frac{z^{\prime}}{2^j}\rfloor
\end{equation}
and 
\begin{equation}\label{donotsatisfy2}
\lfloor \frac{f(z)}{2^{j-1}}\rfloor < \lfloor \frac{f(z^{\prime})}{2^{j-1}}\rfloor.
\end{equation}
By Equation (\ref{donotsatisfy1}), there exist  $z_k, z^{\prime}_k \in Z_{\geq 0}$ for $k =0,1,2,...,n$ such that 
$ z=\sum\limits_{k = 0}^n {{z_k}} {2^k}$ and 
$z^{\prime} = \sum\limits_{k = j}^n {{z_k}} {2^k}+\sum\limits_{k = 0}^{j-1} {{z_k^{\prime}}} {2^k}$.
By the inequality in (\ref{donotsatisfy2}), there exist $y_k, y_k^{\prime} \in Z_{\geq 0}$ for $k =0,1,2,...,n$ and a natural number $i \geq j-1$ such that
$y_i = 0 < 1 = y_i^{\prime}$,
\begin{equation}\label{atheoremforhf}
	f(\sum\limits_{k = 0}^n {{z_k}} {2^k})=  \sum\limits_{k = i+1}^{n} {{y_k}} {2^k}  + y_i \times 2^{i} + \sum\limits_{k = 0}^{i-1} {{y_k}} {2^k}
\end{equation}
and 
\begin{equation}\label{atheoremforhf2}
f(\sum\limits_{k = j}^n {{z_k}} {2^k}+\sum\limits_{k = 0}^{j-1} {{z_k^{\prime}}} {2^k}) = \sum\limits_{k = i+1}^{n} {{y_k}} {2^k}+ y_i^{\prime} \times 2^i + \sum\limits_{k = 0}^{i-1} {{y_k^{\prime}}} {2^k}.
\end{equation}

Let $c =\sum\limits_{k = i+1}^{n} {{y_k}} {2^{k-(i+1)}} $. Then  $c \times 2^{i+1}$  $= \sum\limits_{k = i+1}^{n} {{y_k}} {2^k}$. Hence Equations  (\ref{atheoremforhf}) and (\ref{atheoremforhf2}) imply that 
\begin{equation}\label{atheoremforhfp2}
f(\sum\limits_{k = 0}^n {{z_k}} {2^k})=  c \times 2^{i+1} + 0 \times 2^{i} + \sum\limits_{k = 0}^{i-1} {{y_k}} {2^k}
\end{equation}
and 
\begin{equation}\label{atheoremforhfp2b}
f(\sum\limits_{k = j}^n {{z_k}} {2^k}+\sum\limits_{k = 0}^{j-1} {{z_k^{\prime}}} {2^k}) = c \times 2^{i+1}+ 2^i + \sum\limits_{k = 0}^{i-1} {{y_k^{\prime}}} {2^k}.
\end{equation}
Let 
\begin{equation}\label{definitionofvaa}
a = \max(\{z: f(z) < c \times 2^{i+1} + 2^i\})
\end{equation}
and 
\begin{equation}\label{definitionofvab}
b =  \min(\{z: f(z) \geq  c \times 2^{i+1} + 2^i\}).
\end{equation}
Then,  $b = a+1$ follows directly from (\ref{definitionofvaa}) and (\ref{definitionofvab}).

Let $d =\sum\limits_{k = i+1}^{n} {{z_k}} {2^{k-(i+1)}} $. Then,  $d \times 2^{i+1}$  $= \sum\limits_{k = i+1}^{n} {{z_k}} {2^k}$.
By (\ref{atheoremforhfp2}), $f(d \times 2^{i+1})$ $ \leq f(\sum\limits_{k = 0}^{n} {{z_k}} {2^k})$ $<c \times 2^{i+1}+2^i$, and hence Equation (\ref{definitionofvaa}) implies 
\begin{equation}\label{alargedm2}
a \geq \sum\limits_{k = 0}^n {{z_k}} {2^k} \geq d \times 2^{i+1}.
\end{equation}

By $i+1 \geq j$, we have 
 $(d+1) \times 2^{i+1} $ $=\sum\limits_{k = i+1}^{n} {{z_k}} {2^k}+2^{i+1} $ 
 \begin{equation}\label{starAinequalities}
  > \sum\limits_{k = j}^n {{z_k}} {2^k}+\sum\limits_{k = 0}^{j-1} {{z_k^{\prime}}} {2^k}.
 \end{equation}

By Equation (\ref{atheoremforhfp2b}) and Equqtion (\ref{definitionofvab})
\begin{equation}\label{zkzprimkbigthb}
\sum\limits_{k = j}^n {{z_k}} {2^k}+\sum\limits_{k = 0}^{j-1} {{z_k^{\prime}}} {2^k} \geq b.
\end{equation}
The inequality in (\ref{starAinequalities}) and the inequality in (\ref{zkzprimkbigthb}) imply 
\begin{equation}\label{zkzprimkbigthb2}
(d+1) \times 2^{i+1} > b = a+1.
\end{equation}
The inequality in (\ref{alargedm2}) implies 
\begin{equation}\label{zkzprimkbigthb3}
a+1 >d \times 2^{i+1}.
\end{equation} 

By the inequality in (\ref{zkzprimkbigthb2}) and the inequality in (\ref{zkzprimkbigthb3})
$(d+1) \times 2^{i+1} > b = a+1 > d \times 2^{i+1}$, and hence there exist $d_i$ and $e$ such that $e < 2^i$, $0<d_i 2^i + e$ and 
\begin{equation}\label{expressa1byd2die}
a+1 = d \times 2^{i+1} + d_i2^i + e.
\end{equation}

By Equation (\ref{atheoremforhfp2}) the inequality in (\ref{alargedm2})
\begin{equation}\label{finalcondition1}
f(a) \geq f(\sum\limits_{k = 0}^n {{z_k}} {2^k})  \geq c \times 2^{i+1}.
\end{equation}

Equation (\ref{definitionofvaa}) implies 
\begin{equation}\label{finalcondition2}
f(a) < c \times 2^{i+1} + 2^i
\end{equation}
and 
\begin{equation}\label{finalcondition3}
f(a+1) \geq c \times 2^{i+1} + 2^i.
\end{equation}
Equation (\ref{expressa1byd2die}), the inequality in (\ref{finalcondition1}), the inequality in (\ref{finalcondition2}) and the inequality in (\ref{finalcondition3}) contradict Lemma \ref{asnlemma}. Therefore the function $f$ satisfies the condition of Definition \ref{adefoffunch}.
\end{proof}
By Theorem \ref{necessarycondiforf}, the condition $(a)$ in Definition \ref{adefoffunch} is a necessary condition for $CB(f,y,z)$ to have the 
Grundy number $G(\{y,z\})=y \oplus z$.

\section{A Chocolate Game $CB(f_s,y,z)$ whose Grundy number is ${\bf G_{f_s}(\{y,z\})}$ \\
${\bf = (y \oplus (z+s))-s }$ for a fixed natural number s}
In the previous sections we studied the chocolate bar $CB(f,y,z)$ whose Grundy number is $G(\{y,z\})=y \oplus z$. The condition $G(\{y,z\})=y \oplus z$ is very strong, so we modified it and get the condition that 
$G(\{y,z\}) = (y \oplus (z+s))-s$ for a fixed natural number $s$.

In this section we study a necessary and sufficient condition for a chocolate bar $CB(g,y,z)$ to have the Grundy number $G_g(\{y,z\}) = (y \oplus (z+s))-s$ for a fixed natural number $s$. 

\begin{exam}\label{exampofnimsumchos}
By Lemma \ref{lemmaforfloorzbyk} and Theorem \ref{theoregrundyforf}, the Grundy number of the chocolate bar in Figure \ref{choco832k4ch} is 
\begin{equation}\label{Gfyoplusz}
G_f(\{y,z\})=y \oplus z, 
\end{equation}
where 
\begin{equation}\label{fequalfloortby4}
f(t) = \lfloor \frac{t}{4}\rfloor,
\end{equation}
but Theorem \ref{grundyminuss} (We prove this theorem later in this paper.) implies that the Grundy number of the chocolate bar in Figure \ref{choco823k4p12} is 
\begin{equation}\label{Gfyopluszplus12}
G_{f_{12}}(\{y,z\})=(y \oplus (z+12))-12,
\end{equation}
where 
\begin{equation}\label{fequalfloortby4plus12}
f_{12}(t)=f(t+12)= \lfloor \frac{t+12}{4} \rfloor.
\end{equation}
Please look at the difference between (\ref{Gfyoplusz}) and (\ref{Gfyopluszplus12}), and the 
difference between (\ref{fequalfloortby4}) and (\ref{fequalfloortby4plus12}).

It is easy to see that we get the chocolate bar  in Figure \ref{choco823k4p12} by moving the bitter part horizontally and cutting the chocolate bar  in Figure \ref{choco832k4ch} vertically. We present this method in Figure \ref{choco823k4bb}.

If we generalize this method, we can get a necessary and sufficient condition for a chocolate bar $CB(g,y,z)$ to have the Grundy number $G_g(\{y,z\}) = (y \oplus (z+s))-s$ for a fixed natural number $s$. 

\begin{figure}[!htb]
	\begin{minipage}[!htb]{0.45\columnwidth}
		\centering
		\includegraphics[width=0.9\columnwidth,bb=0 0 278 70]{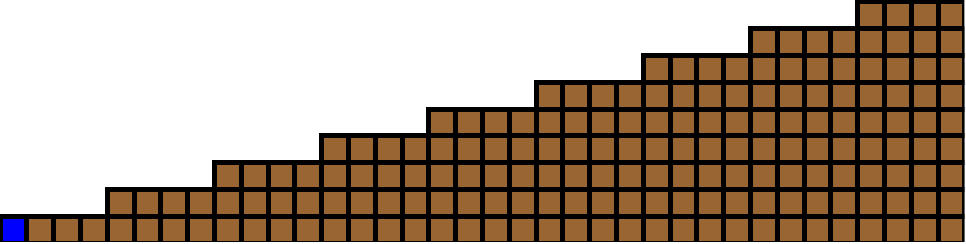}
		\caption{$CB(f,8,32)$ $f(t)$  $= \lfloor \frac{t}{4}\rfloor$.}
		\label{choco832k4ch}
	\end{minipage}
	\begin{minipage}[!htb]{0.45\columnwidth}
		\centering
		\includegraphics[width=0.9\columnwidth,bb=0 0 180 68]{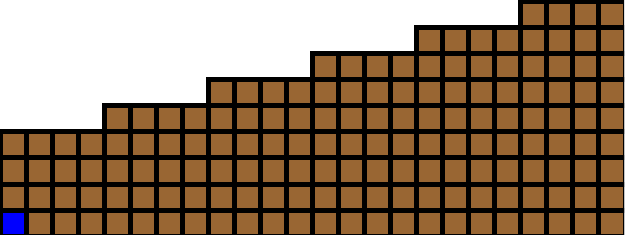}
		\caption{$CB(f_{12},8,23)$ $f_{12}(t)=f(t+12)= \lfloor \frac{t+12}{4} \rfloor$.}
		\label{choco823k4p12}
	\end{minipage}
\end{figure}

\begin{figure}[!htb]
	\centering
	\includegraphics[width=0.6\columnwidth,bb=0 0 361 119]{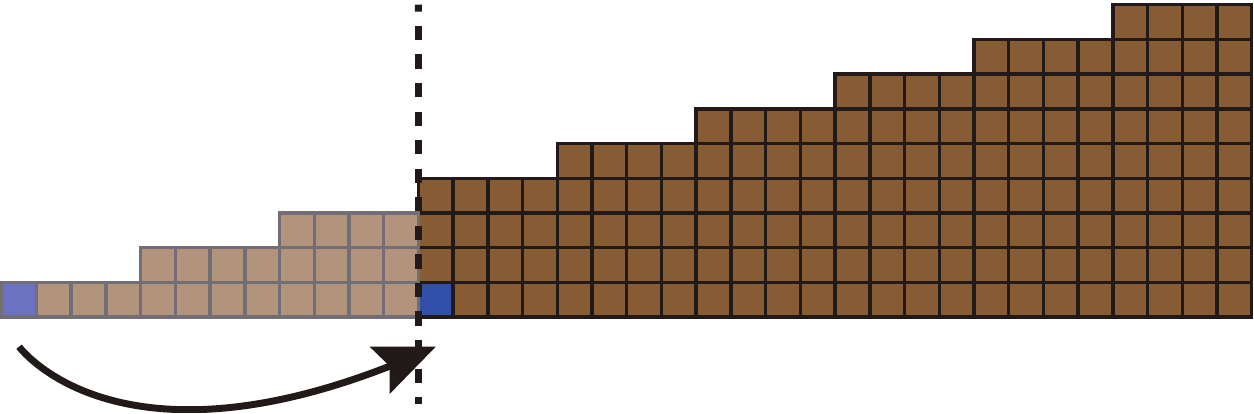}
	\caption{ }
	\label{choco823k4bb}
\end{figure}

\end{exam}

As in Example \ref{exampofnimsumchos} we can make the chocolate bars that have the Grundy number $G_{f_s}(\{y,z\}) = (y \oplus (z+s))-s$ from the chocolate bar whose Grundy number is $G_f(\{y,z\})=y \oplus z$. 

First, we study a sufficient condition.

\subsection{A Sufficient Condition for a Chocolate Bar to have the Grundy Number $(y \oplus (z+s))-s$ }
Let $h$ be the function that satisfies the condition $(a)$ in  Definition \ref{adefoffunch} and let $G_h(\{y,z\})$ be the Grundy number of $CB(h,y,z)$. The condition $(a)$ in Definition \ref{adefoffunch} is a necessary and sufficient condition for $CB(h,y,z)$ to have the Grundy number $G_h(\{y,z\})=y \oplus z$, and
we can use  all the lemmas and theorems in previous sections for the function $h$ and $CB(h,y,z)$.

\begin{defn}\label{defnofsandhs}
We define the function $h_s$ as the followings.\\
Let $s$ be a natural number such that 
	\begin{equation}\nonumber
		i \oplus s = i + s \text{ for } i = 0,1,2,...,h(s).
	\end{equation}
Let $h_s(z) = h(z+s)$ for any $z \in Z_{\geq0}$.
\end{defn}
We are going to show that the condition in Definition \ref{defnofsandhs} is a
necessary and sufficient condition for the chocolate bar $CB(h_s,y,z)$ to have the Grundy number
$G_{h_s}(\{y,z\})= (y \oplus (z+s))-s$.

\begin{Lem}\label{lemmaforshs}
Let $p \in Z_{\geq 0}$ and $s$ be a natural number. Then 
\begin{equation}\label{conditionfoforp}
i \oplus s = i+s  \text{ for } i = 0,1,...,p
\end{equation}
 if and only if 
there exist a natural number $u$ and a non-negative integer $v$ such that
\begin{equation}\label{conditionfoforp2}
s = u \times 2^{v} \text{ and } 2^{v} > p.
\end{equation}
\end{Lem}
\begin{proof}
We suppose Relation (\ref{conditionfoforp}).
When $p > 0$, let $p=\sum\limits_{k = 0}^{t} {{p_k}} {2^k}$ with $p_t=1$ and $p_k \in \{0,1\}$. 
Let $s = \sum\limits_{k = 0}^{n} {{s_k}} {2^k}$ with $s_n=1$ and $s_k \in \{0,1\}$.
Since $p \oplus s = p+s$ and $p_t=1$, $s_t=0$.
Since $0 \leq \sum\limits_{k = 0}^{t-1} {} {2^k} < p$, 
$\sum\limits_{k = 0}^{t-1} {} {2^k} \oplus s = \sum\limits_{k = 0}^{t-1} {} {2^k} + s$. Therefore,  $s_i = 0$ for $i = 0,1,...,t-1$.
Let $u = \sum\limits_{k = t+1}^{n} {s_k} {2^{k-t-1}} $ and $v=t+1$, then we have Relation (\ref{conditionfoforp2}). When $p=0$, we let $v=0$ and $u=s$. Then we have  Relation (\ref{conditionfoforp2}). 

Next we suppose that there exist a natural number $u$ and a non-negative integer $v$ that satisfy  Relation (\ref{conditionfoforp2}). Then it is clear that we have Relation (\ref{conditionfoforp}).
\end{proof}

\begin{Lem}\label{comparnimsumtoarth}
Let $p \in Z_{\geq 0}$ and $s$ be a natural number such that  $i \oplus s = i+s$ for $i = 0,1,...,p$.
Let  $j \in Z_{\geq 0}$ such that $0 \leq j \leq p$. Then 
\begin{equation}\label{thissetequalthatset}
\text{the set } \{j \oplus i:i=0,1,2,...,s-1 \} \text{ is the same as the set }
\{0,1,2,...,s-1\}.
\end{equation}
\end{Lem}
\begin{proof}
By Lemma \ref{lemmaforshs}, there exist a natural number $u$ and a non-negative integer $v$ such that
$s = u \times 2^{v}$ and $2^{v} > p$.

Let $j \in Z_{\geq 0 } $ such that $0 \leq j \leq p$.
Then we write $j$ in base 2, and we have 
$j = \sum\limits_{k = 0}^{v-1} {{j_k}} {2^k}$.
We prove that $0 \leq j \oplus i \leq s-1$ for $0 \leq i \leq s-1$. Let $i \in Z_{\geq 0 } $ such that $0 \leq i \leq s-1$.
Then there exist $u^{\prime} \in Z_{\geq 0}$ such that $u^{\prime} < u$ and $i = u^{\prime} \times 2^{v} + \sum\limits_{k = 0}^{v-1} {{i_k}} {2^k}$ for $i_k \in \{0,1\}$.
Therefore,
\begin{equation}\label{relationji}
0 \leq j \oplus i = u^{\prime} \times 2^{v} + \sum\limits_{k = 0}^{v-1} {{(j_k \oplus i_k)}} {2^k} < u \times 2^{v} = s.
\end{equation}
Since $j \oplus i \neq  j \oplus i^{\prime}$ for $0 \leq i, i^{\prime} \leq s-1$ such that $i \neq i^{\prime}$, the inequality in (\ref{relationji}) implies Relation (\ref{thissetequalthatset}).
\end{proof}

\begin{Lem}\label{lemmaformexs} Let $s$ be a natural number, $x \in Z_{\geq 0}$ and $x_k \in Z_{\geq 0}$ for $k=1,2,...,n$.
Suppose that $x,x_k \geq s$ for $k=1,2,...,n$. Then $\textit{mex}(\{x_k:k=1,2,...,n\} \cup \{0,1,2,...,s-1\}) = x$ if and only if 
$\textit{mex}(\{x_k-s:k=1,2,...,n\} ) = x-s$.
\end{Lem}
\begin{proof}
Suppose that $\textit{mex}(\{x_k:k=1,2,...,n\} \cup \{0,1,2,...,s-1\}) = x$. By Lemma \ref{defofmex1}, $x \notin \{x_k:k=1,2,...,n\} \cup \{0,1,2,...,s-1\}$, and hence  $x-s \notin \{x_k-s:k=1,2,...,n\}$.
Let $u \in Z_{\geq 0}$ such that $u<x-s$. Then $u+s<x$, and Lemma \ref{defofmex1} implies $u+s \in \{x_k:k=1,2,...,n\} \cup \{0,1,2,...,s-1\}$.
Clearly $u+s=x_k$ for some natural number $k$, and hence $u=x_k-s$.
By Lemma \ref{defofmex1}, we have $\textit{mex}(\{x_k-s:k=1,2,...,n\} ) = x-s$. \\
Conversely we suppose that $\textit{mex}(\{x_k-s:k=1,2,...,n\} ) = x-s$. Then, Lemma \ref{defofmex1} implies $x-s \neq x_k -s$ for any $k = 0,1,...,n$, and hence $x \neq x_k$ for any $k = 0,1,...,n$. For any $v \in Z_{\geq 0}$ such that $x > v \geq s$, Lemma \ref{defofmex1} implies that there exists $u$ such that $v = u +s$ and $x-s > u \geq 0$. By Lemma \ref{defofmex1}, there exists $k$ such that $u = x_k -s$, and hence we have $v = x_k$. Therefore Lemma \ref{defofmex1} implies  $\textit{mex}(\{x_k:k=1,2,...,n\} \cup \{0,1,2,...,s-1\}) = x$.
\end{proof}

\begin{Lem}\label{lemmagrundyminuss}
Let $s$ be a natural number such that 
\begin{align}\label{conditionofs1}
i \oplus s = i + s \text{ for } i = 0,1,2,...,h(s).
\end{align}
Then, for any $y,z \in Z_{\geq 0}$ such that $y \leq h(z+s)$, we have
\begin{align}\label{grundyyzpluss}
&y \oplus (z+s) \nonumber \\
&= G_h(\{y,z+s\}) \nonumber\\
&= \textit{mex}(\{v\oplus (z+s):v<y \} \nonumber\\
& \cup \{0,1,2,...,s-1\} \cup \{\min(y,h(w)) \oplus w:s \leq w < z+s\}).
\end{align}
In particular $y \oplus (z+s) \geq s$.
\end{Lem}
\begin{proof}
By Theorem \ref{theoregrundyforf} and the definition of Grundy number,
\begin{align}
y \oplus (z+s) = G_h(\{y,z+s\}) \hspace{70mm}\nonumber\\
=\textit{mex}(\{G_h(\{v,z+s\}):v<y\} \cup \{G_h(\{ \min(y,h(w)),w\}):w < z+s\})\nonumber\\
=\textit{mex}(\{G_h(\{v,z+s\}):v<y\} \cup \{G_h(\{ \min(y,h(w)),w\}):w < s\}\nonumber\\
\cup \{G_h(\{ \min(y,h(w)),w\}):s \leq w < z+s\})\label{lastterm1}.
\end{align}
When $w < s$, we have $h(w) \leq h(s)$, and hence
\begin{align*}
\min(y,h(w)) = \min( \min(y,h(s)),h(w))
\end{align*}

Since $\min(y,h(s)) \in Z_{\geq 0}$ and $\min(y,h(s)) \leq h(s)$, 
Lemma \ref{agrundyequallemma}, Equation (\ref{conditionofs1}) and Lemma \ref{comparnimsumtoarth} imply
\begin{align*}
&\{G_h(\{ \min(y,h(w)),w\}):w < s\} \nonumber\\
& = \{G_h(\{ \min(\min(y,h(s)),h(w)),w\}):w < s\} \nonumber\\
& = \{\min(y,h(s)) \oplus w:w<s\} =\{0,1,2,...,s-1\}.
\end{align*}

 Hence, by Theorem \ref{theoregrundyforf},

(\ref{lastterm1})
\begin{align}
&=\textit{mex}(\{v\oplus (z+s):v<y \} \cup \{0,1,2,...,s-1\} \nonumber \\
&\cup \{ \min(y,h(w)) \oplus w:s \leq w < z+s \}).\label{minorchange11}
\end{align}
Equation (\ref{lastterm1}) and Equation (\ref{minorchange11}) imply Equation (\ref{grundyyzpluss}). Therefore, by Lemma \ref{defofmex2}, we have 
$y \oplus (z+s) \geq s$.
\end{proof}

\begin{Lem}\label{lemmagrundyminuss2}
	Let $s$ be a natural number such that 
	\begin{equation}\label{conditionofs2b}
		i \oplus s = i + s \text{ for } i = 0,1,2,...,h(s).
	\end{equation}
	For any $y,z \in Z_{\geq 0}$ such that $y \leq h(z+s)$, we have
	\begin{align}\label{conditionofstar12}
		(y \oplus (z+s))-s  = \textit{mex}(\{(v\oplus (z+s))-s:v<y \} \nonumber \\
		\cup \{( \min(y,h(w)) \oplus w)-s:s \leq w < z+s \}).
	\end{align}
\end{Lem}

\begin{proof}
	By Relation (\ref{conditionofs2b}) and Lemma \ref{lemmagrundyminuss},  we have
	
	\begin{align}
		& G_h(\{y,z+s\}) \nonumber \\
		&= y \oplus (z+s) \nonumber \\
		&=  \textit{mex}(\{v\oplus (z+s):v<y \} \cup \{0,1,2,...,s-1\} \nonumber \\
		& \hspace{4mm} \cup \{ \min(y,h(w)) \oplus w:s \leq w < z+s \})\nonumber \\
		&= y \oplus (z+s) \nonumber \\
		&=  \textit{mex}(\{v\oplus (z+s):v<y \} \nonumber \\
		& \hspace{4mm} \cup \{0,1,2,...,s-1\} \cup \{ \min(y,h(w^{\prime}+s)) \oplus (w^{\prime}+s):0 \leq w^{\prime} < z \})\label{grundynimussterm}
	\end{align}
	
	for any $y,z \in Z_{\geq 0}$ such that $y \leq h(z+s)$.
	
	Since $v < y \leq h(z+s)$, Lemma \ref{lemmagrundyminuss} implies that for $0 \leq v < y$ 
	\begin{align}\label{conditionof111}
		v\oplus (z+s)\geq s.
	\end{align}
	Since $ \min(y,h(w^{\prime}+s)) \leq h(w^{\prime}+s)$, Lemma \ref{lemmagrundyminuss} implies that for $0 \leq w^{\prime} < z$ 
	\begin{align}\label{conditionof112}
		\min(y,h(w^{\prime}+s)) \oplus (w^{\prime}+s) \geq s.
	\end{align}
Lemma \ref{lemmaformexs}, the inequality in (\ref{conditionof111}), the inequality in (\ref{conditionof112}) and align (\ref{grundynimussterm}) imply (\ref{conditionofstar12}). We have completed the proof.
\end{proof}

\begin{thm}\label{grundyminuss}
	Let $s$ be a natural number such that 
	\begin{align}\label{conditionofs3}
		i \oplus s = i + s \text{ for } i = 0,1,2,...,h(s)
	\end{align}
	and $h_s(z) = h(z+s)$ for any $z \in Z_{\geq0}$.
	Let $G_{h_s}(\{y,z\})$ be the Grundy number of $CB(h_s,y,z)$. Then 
	$G_{h_s}(\{y,z\})= (y \oplus (z+s))-s$ for any $y,z \in Z_{\geq0}$ such that $y \leq h_s(z)$. 
\end{thm}
\begin{proof}
	Let $y,z \in Z_{\geq0}$ such that $y \leq h_s(z)$.
	We prove by mathematical induction, and we assume that $G_{h_s}(\{v,w\})= (v \oplus (w+s))-s$ for $v,w \in Z_{\geq 0}$ such that
	$v \leq y, w< z$ or $v < y, w \leq z$.
	\begin{align}
		&G_{h_s}(\{y,z\})=\textit{mex}(\{G_{h_s}(\{v,z\}):v<y\} \cup \{G_{h_s}(\{ \min(y,h_s(w)),w\}):w < z\})\nonumber \\
		&=\textit{mex}(\{(v \oplus (z+s))-s :v<y\} \cup \{ (\min(y,h(w+s)) \oplus (w+s))-s :w< z \})\nonumber\\
		&=\textit{mex}(\{(v \oplus (z+s))-s :v<y\} \cup \{ (\min(y,h(w+s)) \oplus (w+s))-s\nonumber \\
		& \hspace{90mm} :s \leq w+s < z+s\})\nonumber\\
		&=\textit{mex}(\{(v \oplus (z+s))-s :v<y\} \cup \{ (\min(y,h(w^{\prime})) \oplus w^{\prime})-s :s \leq w^{\prime} < z+s\})\label{gsgrundy1}.
	\end{align}
	By Lemma \ref{lemmagrundyminuss2}, (\ref{gsgrundy1}) = $(y \oplus (z+s))-s$, and hence we finish this proof.
\end{proof}

\subsection{A Necessary Condition for a Chocolate Bar to have the Grundy Number $(y \oplus (z+s))-s$.}
In this subsection, we study a necessary  condition for a chocolate bar  to have the Grundy number $(y \oplus (z+s))-s$.

\begin{defn}\label{defoffuncg}
	Let $s$ be a fixed natural number and $g$ be a function that satisfies the following three conditions:\\
	$(i)$  $g(t)\in Z_{\geq0}$ for $t \in Z_{\geq0}$.\\
	$(ii)$  $g$ is monotonically increasing.\\
	$(iii)$  The Grundy number of $CB(g,y,z)$ is $G_g(\{y,z\})= (y \oplus (z+s))-s$. 
\end{defn} 

We are going to show that there exists a function $h$ such that $g(z) = h(z+s)$ for any $z \in Z_{\geq0}$,
	\begin{align}\nonumber
		i \oplus s = i + s \text{ for } i = 0,1,2,...,h(s),
	\end{align}
and the Grundy number of $CB(h,y,z)$ is $G_h(\{y,z\})= (y \oplus z)$.

\begin{Lem}\label{conditionforg0}
	Let $s$ be a natural number and $g$ a function such that the conditions of Definition \ref{defoffuncg} are satisfied. Then we have 
	$i \oplus s = i + s$ for $i=0,1,2,...,g(0)$.
\end{Lem}
\begin{proof}
	First, we prove that 
	\begin{align}\label{conditionGrundyg}
		G_g(\{i,0\})=i
	\end{align}
	for $i = 0,1,2,...,g(0)$ by mathematical induction. By the definition of Grundy number,
	$G_g(\{0,0\})=0$. We suppose that $G_g(\{k,0\})=k$ for $k = 0,1,2,...i-1$ and 
	$i \leq g(0)$. By the definition of Grundy number,
	$G_g(\{i,0\})=\textit{mex}(\{G_g(\{k,0\}):k=0,1,2,...,i-1\}$ $=\textit{mex}(\{0,1,2,...,i-1\})=i$.
	By the conditions of Definition \ref{defoffuncg}, we have $G_g(\{i,0\})= (i \oplus s)-s$, and hence align (\ref{conditionGrundyg}) implies 
	$(i \oplus s)-s = i$. Therefore, we have completed the proof. 
\end{proof}

\begin{thm}\label{theoremforreverse}
	Let $s$ be a natural number and $g$ a function such that the conditions of Definition \ref{defoffuncg} are satisfied. We define a function $g_{-s}$ by  $g_{-s}(z)=g(z-s)$ for $z \geq s$ and $g_{-s}(z)=g(0)$ for $0 \leq z < s$. 
	Let $G_{g_{-s}}(\{y,z\})$ be the Grundy number of 
	$CB(g_{-s},y,z)$. Then  $G_{g_{-s}}(\{y,z\})= y \oplus z$ for any $y,z \in Z_{\geq 0}$ such that $y \leq g_{-s}(z)$.
\end{thm}
\begin{proof}
	\underline{Case $(1)$} By the definition of $g_{-s}$, we have  $g_{-s}(z)=g(0)$ for $z \leq s$, and hence the function 
	$g_{-s}$ is a constant function for $z \leq s$, and hence it satisfies the condition of Definition \ref{adefoffunch}.
	Therefore $G_{g_{-s}}(\{y,z\})= y \oplus z$ for any $y,z \in Z_{\geq 0}$ such that $y \leq g_{-s}(z)$ and $z \leq s$.\\
	\underline{Case $(2)$} Next we prove that  $G_{g_{-s}}(\{y,z+s\})= y \oplus (z+s)$ for  $y \leq g_{-s}(z+s)$.
	We prove by mathematical induction, and we assume that $G_{g_{-s}}(\{v,w\})= v \oplus w$ for $v,w \in Z_{\geq 0}$ such that
	$v \leq y, w< z+s$ or $v < y, w \leq z+s$. 
	
	By Lemma \ref{conditionforg0} we have $i \oplus s = i + s$ for $i = 0,1,2,...,g(0)$.
	Let $p=g(0)$ and $j= \min(y,g(0))$. Then we use Lemma \ref{comparnimsumtoarth}, and 
 we have Relation (\ref{equal0s}).
	\begin{align}\label{equal0s}
		\text{The set} \{ \min(y,g(0)) \oplus w:w < s\} \text{is the same as the set }
		\{0,1,2,...,s-1\}.
	\end{align}
	By Definition \ref{defoffuncg},
	\begin{align}
		&(y \oplus (z+s))-s=G_g(\{y,z\}) \nonumber \\
		&=\textit{mex}(\{G_g(\{v,z\}):v<y\} \cup \{G_g(\{ \min(y,g(w)),w\}):w < z\})\nonumber \\
		&=\textit{mex}(\{(v \oplus (z+s))-s:v<y\} \nonumber\\
		& \hspace{4mm} \cup \{ (\min(y,g_{-s}(w+s)) \oplus (w+s))-s :s \leq w+s < z+s\})\nonumber \\
		&=\textit{mex}(\{(v \oplus (z+s))-s:v<y\} \cup \{ (\min(y,g(w)) \oplus (w+s))-s :w< z\}).\label{gstoordinaryg}
	\end{align}
	\begin{align}
		&(v \oplus (z+s))-s = G_g(\{v,z\}) \geq 0 \text{ for } v<y \hspace{3cm}\nonumber\\
		&\text{ and }\hspace{8cm} \nonumber \\
		&(\min(y,g_{-s}(w+s)) \oplus (w+s))-s = G_g(\{ \min(y,g(w)),w\}) \geq 0\nonumber\\
		& \text{ for } s \leq w+s < z+s, \nonumber
	\end{align}
	and hence we have
	\begin{align}
		(v \oplus (z+s)) \geq s \text{ for } v<y  \hspace{3cm} \label{eqgrundy1}\\
		\text{ and }\hspace{5.2cm} \nonumber \\
		(\min(y,g_{-s}(w+s)) \oplus (w+s)) \geq s \text{ for } s \leq w+s < z+s. \label{eqgrundy2}
	\end{align}
By the hypothesis of mathematical induction, 
	\begin{align}
		&G_{g_{-s}}(\{y,z+s\}) \nonumber \\
		&=\textit{mex}(\{G_{g_{-s}}(\{v,z+s\}):v<y\}  \cup \{G_{g_{-s}}(\{ \min(y,g_{-s}(w)),w\}):w < s\}\nonumber\\
		& \hspace{4mm} \cup \{G_{g_{-s}}(\{ \min(y,g_{-s}(w)),w\}):s \leq w < z+s\})\nonumber \\
		&=\textit{mex}(\{v \oplus (z+s):v<y\} \cup \{ \min(y,g_{-s}(w)) \oplus w:w < s\} \nonumber \\
		& \hspace{4mm} \cup \{ \min(y,g_{-s}(w)) \oplus w :s \leq w < z+s\}).\label{fromgstog1}
	\end{align}

	Since $g_{-s}(w)=g(0)$ for $0 \leq w < s$, Relation (\ref{equal0s}) implies
	\begin{align*}
		\{ \min(y,g_{-s}(w)) \oplus w:w < s\} = \{ \min(y,g(0)) \oplus w:w < s\}=\{0,1,2,...,s-1\}.
	\end{align*}

Let $x = y \oplus (z+s)$ and $\{x_k, k = 1,2,3,...,n\} = \{v \oplus (z+s):v<y\} \cup \{ \min(y,g_{-s}(w)) \oplus w :s \leq w < z+s\}$. Then 
By align (\ref{gstoordinaryg}), the inequality in (\ref{eqgrundy1}), the inequality in (\ref{eqgrundy2}), align (\ref{fromgstog1}),  we use Lemma \ref{lemmaformexs} and we have
	$G_{g_{-s}}(\{y,z+s\})=y \oplus (z+s)$.
\end{proof}

Theorem \ref{grundyminuss} and Theorem \ref{theoremforreverse} prove the following proposition $(i)$ and $(ii)$ respectively.\\
$(i)$  Let $h$ be a function such that the Grundy number of the chocolate bar $CB(h,y,z)$ is $G_h(\{y,z\}) = y \oplus z$. Then the Grundy number of the chocolate bar $CB(h_s,y,z)$ is $G_{h_s}(\{y,z\}) = (y \oplus (z+s))-s$, where $s$ satisfies the condition (\ref{conditionofs3}) and $h_s(z)=h(z+s)$.\\
$(ii)$ Let $g$ be a function such that the Grundy number of the chocolate bar $CB(g,y,z)$ is $G_g(\{y,z\}) = (y \oplus (z+s))-s$. Then the Grundy number of the chocolate bar $CB(g_{-s},y,z)$ is $G_{g_{-s}}(\{y,z\}) = y \oplus z$, where $g_{-s}(z)=h(z-s)$. Note that $g=(g_{-s})_s$.\\
Therefore we have a necessary and sufficient condition for the chocolate bar $CB(h,y,z)$ to have the Grundy number $G_{h_s}(\{y,z\}) = (y \oplus (z+s))-s$.

Next an example of this condition is presented for the function $h(z)=\lfloor \frac{z}{2k}\rfloor$. As you see, this condition is quite simple for this function.
\begin{Corollary}\label{lemmaforgrundyfloor} Let $h(z)=\lfloor \frac{z}{2k}\rfloor$ for a fixed natural number $k$. Then 
	\begin{align}\label{condiofs}
		s = m2^{v} \text{ for } v,m \in Z_{\geq 0} \text{ such that } m = 0, 1, 2,...,2k-1
	\end{align}
	if and only if 
	the Grundy number of $CB(h_s,y,z)$ is $(y\oplus (z+s)) -s$, where
	$h_s(z) = \lfloor \frac{z+s}{2k}\rfloor$.
\end{Corollary}
\begin{proof} By  Lemma \ref{lemmaforfloorzbyk}, the function $h$ satisfies the conditions of Definition \ref{adefoffunch}.
	By Lemma \ref{lemmaforshs}, 
	\begin{align}
		i \oplus s = i+s  \text{ for } i = 0,1,...,h(s) \nonumber
	\end{align}
	if and only if 
	there exists $u \in Z_{\geq 0}$ such that 
	\begin{align}
		s=u \times 2^v \text{ and } h(s) = \lfloor \frac{s}{2k}\rfloor < 2^v \nonumber
	\end{align} 
	if and only if Condition (\ref{condiofs}) is valid.
	Therefore by Theorem \ref{grundyminuss} we finish the proof of this corollary.
\end{proof}

\end{document}